\documentclass[12pt]{article}
\usepackage{amsthm,amsmath,amsfonts,mathrsfs,amssymb,systeme,bm, bbm}
\usepackage{color,xcolor,graphicx}
\usepackage{natbib}
\usepackage{enumerate}
\usepackage[ruled,vlined]{algorithm2e}

\usepackage{mlmodern}
\usepackage{pdfrender,xcolor}
\usepackage{graphicx,color}
\usepackage[hidelinks,colorlinks=true,linkcolor=blue,citecolor=blue,urlcolor=blue]{hyperref}
\usepackage{algpseudocode}
\usepackage{comment}
\usepackage{cancel}

\newcommand{\R}{\mathbb{R}}
\newcommand{\Z}{\mathbb{Z}}
\newcommand{\N}{\mathbb{N}}
\newcommand{\D}{\mathcal{D}} 

\newcommand{\abs}[1]{\left\lvert#1\right\rvert}
\newcommand{\T}{\prime} 
\newcommand{\up}[1]{{(#1)}}
\newcommand{\e}[1]{e^{\up{#1}}}

\newcommand{\E}{\mathbb{E}}
\newcommand{\J}{\mathcal{J}} 
\renewcommand{\b}[1]{u}
\newcommand{\bb}[2]{u_{#2}}
\newcommand{\ww}[2]{w_{#1#2}}

\newcommand{\route}[1]{\phi_{#1}}
\newcommand{\Route}[2]{\phi_{#1#2}}
\newcommand{\cc}{r} 
\newcommand{\uu}{{\up{\cc}}} 
\newcommand{\fk}{\up{\kappa}}
\newcommand{\I}[1]{\mathbbm{1}{\left( #1 \right)}} 
\newcommand{\II}[1]{\mathbbm{1}{( #1 )}} 
\newcommand{\A}{\mathcal{A}}

\newtheorem{lemma}{Lemma}
\newtheorem{proposition}{Proposition}

\newtheorem{theorem}{Theorem}
\theoremstyle{definition}

\newtheorem{assumption}{Assumption}
\numberwithin{equation}{section} 

\title{Uniform Moment Bounds for Generalized Jackson Networks in Multi-scale Heavy Traffic}

\author{Jin Guang$^1$\\{\small \href{mailto:jinguang@link.cuhk.edu.cn}{jinguang@link.cuhk.edu.cn}} \and Xinyun Chen$^1$\\{\small \href{mailto:chenxinyun@cuhk.edu.cn}{chenxinyun@cuhk.edu.cn}} \and J. G. Dai$^2$\\ {\small \href{mailto:jd694@cornell.edu}{jd694@cornell.edu}}}

\date{{$^1$School of Data Science, The Chinese University of Hong Kong, Shenzhen, China\\%
$^2$School of Operations Research and Information Engineering, Cornell University, Ithaca, New York}} 

\begin{document}

\maketitle
\begin{abstract}
We establish uniform moment bounds for steady-state queue lengths of generalized Jackson networks (GJNs) in multi-scale heavy traffic as recently proposed by \citet{DaiGlynXu2023}. Uniform moment bounds lay the foundation for further analysis of the limit stationary distribution. Our result can be used to verify the crucial moment state space collapse (SSC) assumption in \citet{DaiGlynXu2023} to establish a product-form limit of GJN in the multi-scale heavy traffic regime. Our proof critically utilizes the Palm version of the basic adjoint relationship (BAR) as developed in \cite{BravDaiMiya2023}. 
\end{abstract}

\section{Introduction}
The steady-state queue lengths of a Jackson network with Poisson arrival processes and exponential service times exhibit an attractive product-form distribution \citep{Jack1957,Jack1963}. When interarrival and service time distributions are general, the network is referred to as a generalized Jackson network (GJN). However, the stationary queue length distribution of a GJN does not have an analytic expression in general. In light of this, numerous studies have developed approximations for the stationary distribution of GJNs under certain heavy traffic conditions. \citet{GamaZeev2006} and \citet{BudhLee2009} proved  that, under certain moment-bound conditions on interarrival and service times, the stationary distribution of the GJN weakly converges to that of a semimartingale reflecting Brownian motion (SRBM) when the traffic intensity of each station approaches unity at the same rate. Numerical algorithms for computing the stationary distribution of an SRBM were investigated in \citet{DaiHarr1992} and \citet{ShenChenDaiDai2002}. Recently, \citet{DaiGlynXu2023} introduced the concept of multi-scale heavy traffic, where the traffic intensities of stations approach unity at different rates, resulting in the scaled queue length having a product-form limit for stationary distribution.

In this paper, we show that in multi-scale heavy traffic, the moments of scaled steady-state queue length are uniformly bounded, which verifies the crucial moment state space collapse (SSC) assumption in \citet{DaiGlynXu2023}. In detail, we consider a family of GJNs indexed by $r\in (0,1)$ with $J$ service stations. Under the multi-scale heavy traffic condition, the traffic intensity of station $j$ in the $r$th network is $1-r^j$ for $j=1,\ldots, J$, indicating that different stations approach heavy traffic at different rates. We prove that if interarrival and service times have finite $(M+1)$th moments with $M\geq J$, then the steady-state queue lengths of this family of GJNs are uniformly bounded:
\begin{equation}\label{eq: result}
    \sup_{r\in (0,r_0)} \E\left[\left( r^j Z^{(r)}_j(\infty) \right)^M \right] < \infty, \quad {\text{ for } j=1,\ldots,J.}
\end{equation}
where $r_0$ is a constant in $(0,1)$, and $Z^{(r)}_j(\infty)$ denotes the steady-state queue length at station $j$ in the $r$th network.

Uniform moment bound results, such as \eqref{eq: result}, are essential building blocks in the more sophisticated analysis of the limit stationary distributions of queueing networks in heavy traffic. Assuming exponential integrability for interarrival and service times, \citet{GamaZeev2006} established uniform exponential bounds for the scaled queue length process so that the limit stationary distribution of GJNs in heavy traffic can be characterized via interchange of limits. In \citet{BudhLee2009}, a similar result was obtained by proving uniform moment bounds for the scaled queue length process under a relaxed polynomial integrability assumption. In a recent paper, \citet{BravDaiMiya2023} investigated the limit stationary distribution of multi-class queueing networks under SBP service policies based on the uniform moment bounds established in \citet{CaoDaiZhan2022}. In these works, the uniform moment bounds are proved via the Lyapunov function method introduced in \citet{Dai1995} and \citet{MeynTwee2009}. 

In this paper, we propose a new approach to prove uniform moment bounds by utilizing the basic adjoint relationship (BAR) of the underlying GJNs. A significant benefit of the BAR approach is that it directly characterizes the stationary distribution of queueing networks, eliminating the need to address their transient dynamics.
In a recent series of papers (\citet{BravDaiMiya2017}, \citet{BravDaiMiya2023} and \cite{DaiGlynXu2023}), the authors utilized BAR to establish weak convergence results for GJNs and multi-class queueing networks. The specific BAR equation employed in our proof {was first} derived from \citet{BravDaiMiya2017} for GJN,  with the Palm version further developed in \cite{BravDaiMiya2023}. We believe that our approach can also establish moment bounds for other queueing network models in heavy traffic by utilizing the corresponding BAR equations.

The paper is organized as follows. In Section \ref{subsec:setting}, we review the GJN model and its basic theory and introduce the multi-scale heavy traffic setting; in Section \ref{sec:moment}, we present the main result (Theorem \ref{thm:Ckn}), which establishes a uniform moment bound for the steady-state queue length at different stations under certain moment conditions on the interarrival and service time distributions. The proof of Theorem \ref{thm:Ckn} is based on BAR, which is reviewed in Section \ref{subsec: BAR}. In Section \ref{subsec: test functions }, we describe the proof sketch based on mathematical induction and specify the test functions applied to BAR in the proofs. The detailed proof is presented in Section \ref{sec:proof}.

\section{Problem Setting and Main Results}

\subsection{Generalized Jackson Network and Multi-Scale Heavy Traffic}\label{subsec:setting}
We use $\N$ to denote the set of positive integers $\{1,2,\ldots\}$. For $a,b\in \R$, let $a\wedge b\equiv \min(a,b)$ and $a\vee b\equiv \max(a,b)$. Additionally, $\e{k}$ denotes a $J$-dimensional vector where the $k$th element is $1$ and all other elements are $0$. However, $\e{0}$ represents a $J$-dimensional vector where all elements are $0$.

We consider an open generalized Jackson network (GJN) with $J$ service stations. Each station has a single server that processes jobs one at a time, and it has a buffer with unlimited capacity that holds waiting jobs. Each station receives jobs from an external arrival source (possibly null) and service completions at other stations. When an arriving job at a station finds the server busy, the job waits in the buffer.  Jobs at each station are processed following the first-come-first-serve (FCFS) discipline. When the processing of a job is completed at a station, the job is routed to another station or exits the network. All jobs visiting a station are homogeneous regarding service time distribution and routing probabilities.

Let $\J=\{1, \ldots, J\}$ denote the set of stations. Associated with each station $j\in \J$, there are two i.i.d.~sequences of random variables, $\{T_{e,j}(i);i\in \N\}$ and $\{T_{s,j}(i);i\in \N\}$, two real numbers, $\alpha_j\ge 0$ and $\mu_j>0$, and an i.i.d.~sequence of random vectors $\{\route{j}(i), i\in\N\}$, defined on a common probability space $(\Omega, \mathscr{F}, \mathbb{P})$.
We assume that the $3J$ sequences 
\begin{equation}  \label{eq: 3J sequence}
    \{T_{e,j}(i);i\in \mathbb{N}\}_{j=1}^J,\qquad \{T_{s,j}(i);i\in \mathbb{N}\}_{j=1}^J,\qquad \{\route{j}(i), i\in\N\}_{j=1}^J,
\end{equation}
are independent, and the first $2J$ sequences are unitized, that is, $\E[T_{e,j}{(1)}]=1$ and $\E[T_{s,j}{(1)}]=1$ for all $j\in \J$. 
For each $i\in \N$, $T_{e,j}(i)/\alpha_j$ denotes the interarrival time between the $i$th and $(i+1)$th externally arriving jobs to station $j$, and $T_{s,j}(i)/\mu_j$ stands for the service time of the $i$th job at station $j$. Accordingly, $\alpha_j$ is the external arrival rate, and $\mu_j$ is the service rate at station $j$. The random vector $\route{j}(i)$ represents the routing decision of the $i$th job at station $j$ following the routing probability matrix $P$. Specifically, the job is routed to station $k$ if $\route{j}(i)=\e{k}$ with probability $P_{jk}$, and exits the network if $\route{j}(i)=\e{0}$ with probability $P_{j0}\equiv 1-\sum_{k\in \J}P_{jk}$. We assume that this network is open, i.e., the routing matrix $P$ is assumed to be transient, or equivalently, $I-P$ is invertible, where $I$ is the identity matrix.

\paragraph{Markov process.} A GJN can be modeled as a Markov process as follows. For time $t \geq 0$, we denote by $Z_{j}(t)$ the number of jobs at station $j$, including possibly the one in service. Let $R_{e, j}(t)$ be the residual time until the next external arrival to station $j$, and $R_{s, j}(t)$ be the residual service time for the job being processed at station $j$ if $Z_j(t) > 0$ or the service time of the next job to be processed at station $j$ if $Z_j(t) = 0$. We write $Z(t)$, $R_{e}(t)$ and $R_{s}(t)$ for $J$-dimensional random vectors whose $j$th element are $Z_{j}(t)$, $R_{e, j}(t)$ and $R_{s, j}(t)$, respectively. For any $t\geq 0$, set
\begin{equation*}
    X(t)=\big( Z(t),  R_{e}(t),  R_{s}(t)\big).
\end{equation*}
Then $\{X(t) , t \geq 0\}$ is a Markov process with respect to the filtration $\mathbb{F}^{X} \equiv \left\{\mathcal{F}_{t}^{X} ; t \geq 0\right\}$ defined on the state space $\mathbb{S} = \Z_{+}^{J} \times \R_{+}^{J} \times \R_{+}^{J}$, where $\mathcal{F}_{t}^{X}=\sigma(\{ X(s) ; 0 \leq s \leq t\})$. We assume that each sample path of $\{X(t) , t \geq 0\}$ is right-continuous and has left limits. 

\paragraph{Traffic equation.} Let $\{\lambda_{j}, j\in \J\}$ be the solution to the traffic equations:
\begin{equation}\label{eq:traffic}
    \lambda_{j}=\alpha_{j}+\sum_{\ell\in \J} \lambda_\ell P_{\ell j}, \qquad \forall~j\in \J.
\end{equation}
For each $j\in\J$, $\lambda_j$ is referred to as the nominal total arrival rate to station $j$, which consists of both external arrivals and internal arrivals from other stations. 
The traffic intensity at station $j$ is given by $\rho_j\equiv{\lambda_j}/{\mu_j}$. 

\paragraph{Multi-scale heavy traffic.} Following \citet{DaiGlynXu2023},  we consider a sequence of GJNs indexed by $\cc \in(0,1)$, and denote by $\{X^\uu(t), t \geq 0\}$ the corresponding Markov process. We write $\mu_j^\uu$ for the service rate at station $j$ in the $\cc$th GJN. Note that the service rates are the only model parameters dependent on $\cc$. All other parameters are assumed to be independent of $\cc$, including the external arrival rates denoted by $\{\alpha_j\}_{j=1}^J$, unitized interarrival and service times, as well as the routing vectors specified in \eqref{eq: 3J sequence}.

\begin{assumption}[Multi-scale heavy traffic] \label{assmpt: multiscale}
    We assume that for all $r\in (0,1)$
    \begin{equation*}
        \mu_j^\uu - \lambda_j = \cc^j, \quad \forall~j\in \J.
    \end{equation*}
\end{assumption}
Under Assumption \ref{assmpt: multiscale}, $\rho^{(r)}_j\equiv \lambda_j/\mu^{(r) }_j \to 1$ for all $j\in \J$. In other words, each station is in heavy traffic when $r\to 0$. But different stations approach heavy traffic at different rates. \citet{Dai1995} proves that the following assumption holds under some mild distributional conditions on interarrival time distributions. 

\begin{assumption}[Stability]
    For each $\cc \in(0,1)$, the Markov process $\{X^\uu(t), t \geq 0\}$ is positive Harris recurrent and has a unique stationary distribution $\pi^\uu$.
\end{assumption}
For $r\in(0,1)$, we denote by 
$$X^\uu=\left(Z^\uu,R_{e}^\uu,R_s^\uu\right)$$
the random vector that follows the stationary distribution. To shorten the notation, we use $\E_{\pi}[\cdot]$ (rather than $\E_{\pi^{(r)}}[\cdot]$) to denote expectation with respect to the stationary distribution when the index $r$ is clear from the context.

\subsection{Main Results: Uniform Steady-State Moment Bounds}\label{sec:moment}

Our main result is a uniform bound on the steady-state moment of the queue length in GJNs when the multi-scaling parameter $r$ is small enough, under certain moment conditions on the unitized interarrival and service times.

\begin{theorem}\label{thm:Ckn}
    For a given integer $M\geq 1$, suppose the following moments exist for the unitized times:  
    \begin{equation}\label{eq: moment condtion}
        \E\left[T_{e,j}^{ M+1  }{(1)}\right]<\infty  \text{ for } 1\leq j \leq M \wedge J, \text{ and }~\E\left[T_{s,j}^{ M+1}{(1)}\right]<\infty\text{ for  }1\leq j \leq J.
    \end{equation}
    Then, for each $1\leq k\leq M\wedge J$, there exists a positive constant $C_{k}<\infty$ such that for all $\cc\in(0,r_0)$,
    \begin{equation*} 
        \E_\pi\left[ \left( \cc^k Z_{k}^\uu  \right)^{M}\right] \leq C_{k},
    \end{equation*}
    where $r_0\in (0,1)$ is a constant independent of $M$.
\end{theorem}

The constant $r_0$ is specified in Lemma \ref{lmm: eta} in Section \ref{subsec: proof of S1}. Its value is determined by the routing matrix $P$ of the GJN. The moment bound in Theorem \ref{thm:Ckn} can be generalized to cases where $M$ is a non-integer. Theorem~\ref{thm:general} in Section~\ref{sec: beta general} is an extension of Theorem~\ref{thm:Ckn} with  $M=J+\varepsilon$ for some $\varepsilon\in (0,1)$.  The uniform moment bound in Theorem~\ref{thm:general} lays a foundation in \cite{DaiGlynXu2023} to prove asymptotic product-form steady-state distribution of GJNs in multi-scale heavy traffic; see Proposition 3.1 of \cite{DaiGlynXu2023}, which is a re-statement of Theorem~\ref{thm:general}.

\section{Sketch of Proof}
Our proof approach is rooted in the BAR equation in \cite{BravDaiMiya2017,BravDaiMiya2023}, which enables us to directly characterize the stationary distribution of the Markov processes $\{X^\uu(t), t \geq 0\}$. Specifically, we utilize BAR with appropriately designed test functions to establish relations between the higher-order and lower-order moments of the queue length $Z^\uu_k$. Through these relations, we apply a mathematical induction argument to prove Theorem \ref{thm:Ckn}. In Section \ref{subsec: BAR}, we first provide a brief overview of the specific BAR of the GJN that is used in our analysis. In Section \ref{subsec: test functions }, we present a proof sketch for Theorem \ref{thm:Ckn}, where we specify the induction hypotheses and the test functions used in the proof. The technical lemmas and proof details are presented in Section \ref{sec:proof}.

\subsection{Basic Adjoint Relationship}\label{subsec: BAR}
Let $\D$ be the set of bounded function $f: \mathbb{S} \rightarrow \mathbb{R}$ satisfying the following conditions: for any $z \in \mathbb{Z}_{+}^J$, (a) the function $f(z,\cdot,\cdot):\R_+^{J}\times \R_+^J\rightarrow \R$ is continuously differentiable at all but finitely many points; (b) the derivatives of $f(z,\cdot,\cdot)$ in each dimension have a uniform bound over $z$.

For a GJN whose Markov process $\{X(t),t\geq 0\}$ is defined in Section \ref{subsec:setting}, we denote by $\pi$ its stationary distribution and by $X$ the steady-state random vector. In our proof, we shall utilize the BAR for the GJN as proposed in \citet{BravDaiMiya2017}:
\begin{equation}\label{eq:BAR17}
    \begin{aligned}
        \E_{\pi}\left[  \A f(X) \right]&+\sum_{j\in \J}\E_{\pi}\left[\sum_{m=1}^{\infty} \left( f(X_{\delta^{e,j}_{m}+}) - f(X_{\delta^{e,j}_{m}-}) \right)  \I{0<t_m^{e,j} \leq 1 } \right]  \\
        &+\sum_{j\in \J}\E_{\pi}\left[\sum_{m=1}^{\infty} \left( f(X_{\delta^{s,j}_{m}+}) - f(X_{\delta^{s,j}_{m}-}) \right)  \I{0<t_m^{s,j} \leq 1 } \right]  = 0,\quad  \forall~ f\in\mathcal{D}.
    \end{aligned}
\end{equation}
For each $f\in\mathcal{D}$, define the ``interior operator''
\begin{equation*}
    \A f(x)=-\sum_{j \in \J} \frac{\partial f}{\partial r_{e,j}}( x)-\sum_{j \in \J} \frac{\partial f}{\partial r_{s,j}}( x) \I{z_j>0},\quad x=(z,r_{e},r_{s})\in \mathbb{S},
\end{equation*}
where the first term on the right-hand side represents the passage of residual times of external arrivals, and the second term stands for the passage of residual service times, which happens only when the corresponding station is busy. In this view, $\A f$ corresponds to the evolution of the GJN between jumps.

In the second and third terms in \eqref{eq:BAR17}, $\delta^{e,j}_m$ and $\delta^{s,j}_m$ represent  the $m$th external arrival and service completion at station $j$, respectively, for $j\in \mathcal{J}$ and $m\in \N$. Accordingly, we use $X_{\delta_{m }^{e,j}-}$ and $X_{\delta_m^{e,j}+}$ to denote the states of the process just before and immediately after the $m$th external arrival to station $j$, and similarly use $X_{\delta_{m}^{s,j}-}$ and $X_{\delta_m^{s,j}+}$ to represent the states of the process just before and immediately after the $m$th service completion at station $j$. The terms $t^{e,j}_m$ and $t^{s,j}_m$ are the times at which the corresponding jumps occur. Therefore, the second and third terms in \eqref{eq:BAR17} correspond to state changes {caused} by jumps of external arrivals and service completions, respectively. 

\vspace{1em}

To simplify the expression in \eqref{eq:BAR17}, we utilize the concept of the Palm measure introduced in \citet{BravDaiMiya2023}. 
Specifically, we introduce expectation operators $\mathbb{E}_{e, j}$ and $\mathbb{E}_{s, j}$ on a pair of random variables $\left(X_-, X_{+}\right) \in$ $\mathbb{S}^2$ and Borel measurable functions $f,g\in \D$ as follows:
\begin{equation}\label{eq: Ee}
    \mathbb{E}_{e, j}\left[f\left(X_- \right) g\left(X_{+}\right)\right]=\frac{1}{\alpha_j} \mathbb{E}_{\pi}\left[\sum_{m=1}^{\infty} f(X_{\delta^{e,j}_{m}-}) g(X_{\delta^{e,j}_{m}+})\I{0<t_m^{e,j} \leq 1 }\right], \quad j\in \J,
\end{equation}
\begin{equation}\label{eq: Es}
    \mathbb{E}_{s, j}\left[f\left(X_- \right) g\left(X_{+}\right)\right]=\frac{1}{\lambda_j} \mathbb{E}_{\pi}\left[\sum_{m=1}^{\infty} f(X_{\delta^{s,j}_{m}-}) g(X_{\delta^{s,j}_{m}+}) \I{0<t_m^{s,j} \leq 1 }\right], \quad j\in \J.
\end{equation}The results (A.12) and (A.13) in \citet{BravDaiMiya2017} demonstrate that $\mathbb{E}_{e, j}[1]=1$ and $\mathbb{E}_{s, j}[1]=1$. Therefore, the expectations $\mathbb{E}_{e, j}[\cdot]$ and $\mathbb{E}_{s, j}[\cdot]$ defined in \eqref{eq: Ee} and \eqref{eq: Es} are probability measures, which we denote by $\mathbb{P}_{e, j}$ and $\mathbb{P}_{s, j}$, respectively. Hence, the distribution of $\left(X_-, X_{+}\right)$ under $\mathbb{P}_{e, j}$ (resp.~$\mathbb{P}_{s, j}$) is determined by \eqref{eq: Ee} (resp.~\eqref{eq: Es}). Consequently, we can consider $X_-$ as the pre-jump state for each external arrival or service completion and $X_{+}$ as the post-jump state. 
\begin{lemma}\label{lmm: independent}
    The pre-jump state $X_{-}$ and the post-jump state $X_{+}$ have the following representation,
    $$
    \begin{aligned}
    & X_{+}=X_{-}+\left( \e{j},   \e{j}T_{e, j} /\alpha_j, 0 \right), \quad \text { under } \mathbb{P}_{e, j}, j \in \J, \\
    & X_{+}=X_{-}+\left( - \e{j}+ \route{j} , 0,   \e{j}T_{s, j} /\mu_j \right), \quad \text { under } \mathbb{P}_{s, j}, j \in \J,
    \end{aligned}
    $$
    where $T_{e, j},T_{s, j}, \route{j}$ for $j \in \J$ are random variables defined on the measurable space $(\mathbb{S}^2, \mathscr{B}(\mathbb{S}^2))$, such that, under Palm distribution $\mathbb{P}_{e, j}$, $T_{e, j}$ is independent of $X_{-}$ and has the same distribution as that of $T_{e, j}(1)$ on $(\Omega, \mathscr{F}, \mathbb{P})$, and, under Palm distribution $\mathbb{P}_{s, j}$, $\left(T_{s, j}, \route{j}\right)$ is independent of $X_{-}$ and has the same distribution as that of $\left(T_{s, j}(1), \route{j}(1)\right)$ on $(\Omega, \mathscr{F}, \mathbb{P})$.
\end{lemma}
\begin{proof}[Proof of Lemma \ref{lmm: independent}]
    The proof follows from Lemma~6.3 in \citet{BravDaiMiya2023}.
\end{proof}
To simplify the notation, we shall omit the subscript and denote by $X\equiv (Z,R_e, R_s) $ the pre-jump state in the Palm expectations defined in \eqref{eq: Ee} and \eqref{eq: Es}. The increments corresponding to different jump events are denoted by
\begin{equation*}
    \Delta_{e,j} \equiv \left( \e{j},   \e{j}T_{e, j} /\alpha_j, 0 \right)~\text{ and }~  \Delta_{s,j} \equiv \left( -\e{j}+ \route{j} , 0,    \e{j}T_{s, j} /\mu_j \right), \qquad j\in \J.
\end{equation*} 
Then, the BAR introduced by \eqref{eq:BAR17} becomes: for any $f\in \D$,
\begin{align}\label{eq:BAR}
    -\E_{\pi}\left[  \A f(X ) \right]
    = \sum_{j\in \J} \alpha_j \E_{e,j} \left[   f(X+{\Delta}_{e,j}) - f(X)  \right] + \sum_{j\in \J} \lambda_j \E_{s,j} \left[   f(X+{\Delta}_{s,j}) -  f(X)  \right].
\end{align}
{To derive} \eqref{eq:BAR}, we assume that there are no simultaneous events among external arrivals and service completions, where event times $t^{e,j}_m$ and $t^{s,j}_m$ are used in \eqref{eq: Ee} and \eqref{eq: Es}. However, BAR \eqref{eq:BAR} still holds even if there are simultaneous events, as detailed in Section 6 of \cite{BravDaiMiya2023}.

\subsection{Sketch of Proof and Test Functions}\label{subsec: test functions }
In this section, we present {the} proof sketch for Theorem \ref{thm:Ckn} using mathematical induction. Our induction hypotheses include moment bounds for the queue length and auxiliary results bounding the expectations of some cross terms of the queue length and the residual interarrival or service times.

\paragraph{Induction hypotheses.} Given the integer $M\geq 1$, for each integer pair $(k,n)$ with  $1\leq k\leq M\wedge J$ and  $0\leq n\leq M$, there exist positive and finite constants $C_{k,n}, D_{k,n}, E_{k,n}, F_{k,n}$ that are independent of $\cc$ such that the following statements hold for all $\cc\in (0,r_0)$, where the constant $r_0$  will be specified later in Lemma \ref{lmm: eta}: 
\begin{enumerate}
    \item[\hypertarget{S1}{(S1)}] ~$
    \E_{\pi}\left[\left(\cc^kZ_k^\uu\right)^n\right]\leq C_{k,n}$;
\item[\hypertarget{S2}{(S2)}] 
$\sum\limits_{\ell=1}^{M\wedge J}\E_{e,\ell}\left[\left(\cc^kZ_k^\uu\right)^n\right]+ \sum\limits_{\ell=1}^J\E_{s,\ell}\left[\left(\cc^kZ_k^\uu\right)^n\right] \leq D_{k,n}$;
\item[\hypertarget{S3}{(S3)}] 
    ~$\E_{\pi} \left[ \left( \cc^k Z_{k}^{\uu}\right)^{n}\psi_{M-n}\left(R_e^{(r)},R_s^{(r)}\right) \right] \leq E_{k,n}$;
\item[\hypertarget{S4}{(S4)}]  
$\sum\limits_{\ell=1}^{M\wedge J} \E_{e,\ell}\left[ \left( \cc^k Z_{k}^{\uu}\right)^{n} \psi_{M-n}\left(R_e^{(r)},R_s^{(r)}\right)\right]+\sum\limits_{\ell=1}^J \E_{s,\ell}\left[ \left( \cc^k Z_{k}^{\uu}\right)^{n} \psi_{M-n}\left(R_e^{(r)},R_s^{(r)}\right)\right] \leq F_{k,n}$;
\end{enumerate}
where 
\begin{equation}\label{eq: function psi}
    \psi_{n}\left(r_e, r_s \right) = \sum_{j=1}^{M\wedge J} r_{e,j}^{n} +  \sum_{j=1}^J  r_{s,j}^{n}.
\end{equation}
It is essential to note that the function $\psi_{M-n}$, appearing on the left-hand side of the auxiliary statements (\hyperlink{S3}{S3}) and (\hyperlink{S4}{S4}), depends on the order $M+1$ of the moment condition~\eqref{eq: moment condtion} on the unitized times. This design of the auxiliary statements plays a crucial role in our proof and assists in reducing the order of the moment condition \eqref{eq: moment condtion} required for establishing uniform bounds on $\E_{\pi}[(\cc^kZ_k^\uu)^M]$ for $1\leq k \leq M\wedge J$.

\paragraph{Mathematical induction.}
We initially prove Statements (\hyperlink{S1}{S1})-(\hyperlink{S4}{S4}) for the base step, i.e., when $1\leq k\leq M\wedge J$ and $n=0$. The proof details are provided in Section \ref{subsec: bounds for residual times}. For the inductive step concerning each pair $(k,n)$, we follow the sequence $(1,1),(1,2),\ldots,(1,M)$, $(2,1),(2,2),\ldots,(2,M),\ldots,(M\wedge J,1),(M\wedge J,2),\ldots,(M\wedge J,M)$. Specifically, we verify Statements (\hyperlink{S1}{S1})-(\hyperlink{S4}{S4}) for each given pair $(k,n)$, under the induction hypotheses that they are true for all pairs in $S_{k,n}$, defined as
\begin{equation*}
    S_{k,n}\equiv \left\{ (i,m): 1\le i\le k-1, 0\le m \le M \text{ or } i=k, 0\le m\le n-1\right\}.
\end{equation*}
To prove each statement, we employ the BAR \eqref{eq:BAR} along with a properly designed test function, which allows us to bound the {statements} of the pair $(k,n)$ by those of the pairs in $S_{k,n}$, in accordance with the induction hypotheses. 

Before presenting the proof sketch and test functions for each statement, we first define a matrix  {$(w_{jk})_{j,k\in\J}$} in Lemma \ref{lmm: w}. This matrix is chosen based on the routing probabilities of the GJN and will be used in the subsequent proofs.
\begin{lemma}\label{lmm: w} 
The following set of equations 
$$\quad w_{jk} = P_{jk} + \sum_{\ell=1}^{ k-1}P_{j\ell}w_{\ell k}, \quad\forall~ j, k\in \J,$$
has a unique solution $(w_{jk})\in [0,1]^{J\times J}$, and $w_{kk}<1$ for any $k\in \J$.
\end{lemma}
\begin{proof}[Proof of Lemma \ref{lmm: w}]
    The proof of existence and uniqueness follows from Lemma 7.3 in \cite{DaiGlynXu2023}. For any $j,k\in \J$, $w_{jk}$ can be interpreted as the probability that a customer starting from station $j$ will enter station $k$ before exiting the network or visiting any stations in $\{k+1,...,d\}$. Since the GJN is open, $w_{kk}$ cannot be $1$.
\end{proof}

\paragraph{Proof of (S1).} To prove (\hyperlink{S1}{S1}) for the pair $(k,n)$, we consider the following test function:
\begin{equation}\label{eq:funcS1o}
	f _{k,n}\left(x\right) = \frac{1}{n+1} r^{k(n-1)} \left( \b{k}^\T z\right)^{n+1}  + r^{k(n-1)}\left( \b{k}^\T z\right)^{n} h _{k}\left(r_e, r_s \right),\quad x=(z,r_{e},r_{s})\in \mathbb{S},
\end{equation}
where the vector $u$ is given by 
\begin{equation}\label{eq:u}
    u=(w_{1k},\cdots, w_{k-1,k}, 1, 0, \ldots, 0)\in \R^J_+,
\end{equation}
and
\begin{equation}\label{eq:h}
	h_{k}\left(r_e, r_s \right) =  -  \sum_{j=1}^k  \bb{k}{j}  \alpha_j{r}_{e,j} +    {\mu}^\uu_k {r}_{s,k}  -\sum_{j=k}^{J} w_{jk} {\mu}^\uu_j r_{s,j}.
\end{equation}
The function $h_k$ is designed such that it contains only the residual interarrival times of lighter-traffic stations and residual service times of heavier-traffic stations, and the coefficients are chosen according to the routing probabilities of the GJN, as illustrated in Lemma \ref{lmm: w}. This will ensure that applying the operator $\A$ to $h_k$ yields terms of order $r^k$ (see Lemma \ref{lmm: h simplify}) and the Palm terms in BAR \eqref{eq:BAR} will have the polynomial in $u'z$ with order up to $n-1$.

By substituting \eqref{eq:funcS1o} into the BAR \eqref{eq:BAR}, we obtain an inequality in the form of 
$$\E_{\pi}\left[ a \left(\cc^{k}\b{k}^\T {Z}^{\uu}\right)^n\right]\leq B,$$ 
where $a>0$ is a constant independent of $r$. The constant $B$ is {the upper bound of} a linear combination of terms given by Statements (\hyperlink{S1}{S1})-(\hyperlink{S4}{S4}) corresponding to pairs in $S_{k,n}$ with finite coefficients that are independent of $r$. Thus, $B$ is also finite and independent of $r$ following the induction hypotheses. Consequently, we can conclude that $$\E_{\pi}\left[\left(\cc^kZ_k^\uu\right)^n\right]\leq \E_{\pi}\left[\left(\cc^{k} \b{k}^\T {Z}^{\uu}\right)^n \right] \leq B/a.$$ 
The complete proof of (\hyperlink{S1}{S1}), including the definition of constant $a$ and the derivation of $B$, is given in {\eqref{eq: result S1} of} Section \ref{subsec: proof of S1} with detailed calculations.

\paragraph{Proofs of (S2) to (S4)}
The proofs for Statements (\hyperlink{S2}{S2})-(\hyperlink{S4}{S4}) follow a similar argument as in the proof of (\hyperlink{S1}{S1}), in which {statements} are bounded by applying specific test functions to the BAR \eqref{eq:BAR} and utilizing the induction hypotheses. Below we specify the test functions for (\hyperlink{S2}{S2})-(\hyperlink{S4}{S4}) in \eqref{eq:funcS2o}-\eqref{eq:funcS4o}, respectively. The complete proofs are given in Sections \ref{subsec: S2}-\ref{subsec: S4}.

\begin{equation}\label{eq:funcS2o}
    f_{k,n,D}\left(x\right) =  \left(\cc^{k} {z}_k\right)^n \psi_{1}\left(r_e,r_s\right),\quad x=(z,r_{e},r_{s})\in \mathbb{S},
\end{equation}

\begin{equation}\label{eq:funcS3o}
    f_{k,n,E}\left(x\right) = \left( \cc^{k} {z}_k\right)^n   \psi_{M-n+1}\left(r_e,r_s\right),\quad x=(z,r_{e},r_{s})\in \mathbb{S},
\end{equation}

\begin{equation}\label{eq:funcS4o}
    f_{k,n,F}\left(x\right) = \left(\cc^{k}z_k\right)^n \psi_{M-n} \left(r_e,r_s\right) \psi_{1}\left(r_e,r_s\right),\quad x=(z,r_{e},r_{s})\in \mathbb{S},
\end{equation}
where the functions $\psi_1$, $\psi_{M-n+1}$ and $\psi_{M-n}$ are {as} defined in \eqref{eq: function psi}.

\section{Proof Details}\label{sec:proof} 
In this section, we present {the} detailed proof of Theorem \ref{thm:Ckn} using mathematical induction. We begin by establishing the base step in Section \ref{subsec: bounds for residual times} {for $1\leq k\leq M\wedge J$ and $n=0$}. Then, in Sections \ref{subsec: proof of S1}-\ref{subsec: S4}, we carry out the inductive steps for Statements (\hyperlink{S1}{S1})-(\hyperlink{S4}{S4}). However, the test functions in Sections \ref{subsec: bounds for residual times}-\ref{subsec: S4} are not bounded and hence not in $\D$. In Section \ref{sec: truncation}, we make the proofs in Sections \ref{subsec: bounds for residual times}-\ref{subsec: S4} rigorous by introducing the truncated test functions, which are in $\D$.
Finally, in Section \ref{sec: beta general}, we extend the main results to moments of non-integer order. Throughout the rest of the paper, we use the shorthand notation $\psi$ for $\psi(R_e^\uu,R_s^\uu)$.

\subsection{Proof of the Base Step When \texorpdfstring{$n=0$}{n=0}}\label{subsec: bounds for residual times}
In this section, we aim to demonstrate that Statements (\hyperlink{S1}{S1})-(\hyperlink{S4}{S4}) hold for any station $1\leq k \leq M\wedge J$ when $n=0$. Clearly, (\hyperlink{S1}{S1}) and (\hyperlink{S2}{S2}) are trivially satisfied for $n=0$. When $n=0$, (\hyperlink{S3}{S3}) holds for $M=1$ by Lemma~4.5
of \citet{BravDaiMiya2017}. For $M\ge 1$ with $n=0$, Lemma~6.4 of \citet{BravDaiMiya2023} proves that for any $j\in \mathcal{J}$,
\begin{align*}
  \sup_{r\in (0,1)}  \E_{\pi} \left[ \left( R^\uu_{e,j} \right)^M \right] <\infty, \quad
  \sup_{r\in (0,1)}  \E_{\pi} \left[ \left( R^\uu_{s,j} \right)^M \I{Z_j^{(r)}>0}\right] <\infty.
\end{align*}
Since $\E_{\pi} [ ( R^\uu_{s,j} )^M \II{Z_j^{(r)}=0}] = \E [ (T_{s,j}/\mu^\uu_j)^M \II{Z_j^{(r)}=0}]{<\infty}$ for any $j\in \mathcal{J}$ according to Lemma~\ref{lmm: independent}, we have
\begin{align}
    \sup_{r\in (0,1)}  \E_{\pi} \left[ \left( R^\uu_{s,j} \right)^M \right] &=\sup_{r\in (0,1)}  \left( \E_{\pi} \left[ \left( R^\uu_{s,j} \right)^M \I{Z_j^{(r)}>0}\right] +   \E_{\pi} \left[ \left( R^\uu_{s,j} \right)^M \I{Z_j^{(r)}=0}\right] \right) \notag \\
    &\equiv E_{k,0} <\infty, \label{eq: E10}
\end{align}
which implies that (\hyperlink{S3}{S3}) holds for $M\ge 1$, $n=0$ and $1\leq k\leq M\wedge J$, with $E_{k,0}$ as given above.

To prove Statement (\hyperlink{S4}{S4}) when $n=0$, we employ the test function:
\begin{equation}\label{eq:func base}
    f_{0,F} \left(x\right) = \psi_{M}(r_e,r_s)  \psi_{1}(r_e,r_s),\quad x=(z,r_{e},r_{s})\in \mathbb{S}.
\end{equation}
Substituting $f_{0,F}$ into the BAR \eqref{eq:BAR}, the left-hand side becomes
\begin{align}
    -\E_{\pi}\left[\A f_{0,F} \left(X^\uu\right) \right] &=M \E_{\pi}\left[  \left( \sum_{j=1}^{M\wedge J} \left( R_{e,j}^\uu \right)^{M-1}  +\sum_{j=1}^J \left( R_{s,j}^\uu \right)^{M-1} \I{Z_{j}^{\uu}>0} \right)   \psi_{1}  \right] \notag \\
    &\quad + \E_{\pi}\left[    \left( {M\wedge J}  +\sum_{j=1}^J \I{Z_{j}^{\uu}>0} \right) \psi_{M}    \right]\notag \\
    &\leq M \E_{\pi}\left[ \psi_{M-1}   \psi_{1}  \right]  + 2J\E_{\pi}\left[ \psi_{M}   \right]  \leq \left( 4J^2 M + 2J \right) E_{1,0}, \label{eq: base step LHS}
\end{align}
where the {last} inequality holds due to the base step of Statement (\hyperlink{S3}{S3}) and {the fact that}
\begin{align}\label{eq: cross term bound}
    \psi_{M-1}  \psi_{1}  \leq 2J \left( R_{\max}^\uu \right)^{M-1}   2J R_{\max}^\uu = 4J^2 \left( R_{\max}^\uu \right)^{M} \leq 4J^2 \psi_{M},
\end{align}
where $R_{\max}^\uu\equiv \max(\{R_{e,j}^{(r)}: 1\leq j\leq M\wedge J\}\cup\{R_{s,j}^{(r)}: 1\leq j\leq J\})$.

For each $\ell \leq M\wedge J$, the term on the right-hand side of the BAR that corresponds to jumps of external arrivals to station $\ell$ becomes
\begin{align}
    \E_{e,\ell}\left[f_{0,F} \left(X^\uu+ \Delta_{e,\ell}\right) - f_{0,F} \left(X^\uu \right)  \right] &= \E_{e,\ell} \left[\left(\psi_{M}    +   \left(   \frac{T_{e,\ell}}{\alpha_{\ell}}  \right)^{M}    \right)\left( \psi_{1}   +     \frac{T_{e,\ell}}{\alpha_{\ell}} \right) - \psi_{M} \psi_{1}  \right] \notag \\
    & \geq   \E_{e,\ell}\left[  \psi_{M}  \frac{T_{e,\ell}}{\alpha_{\ell}}  \right] = \frac{\E_{e,\ell}\left[  \psi_{M}  \right] }{\alpha_{\ell}}  \label{eq: base step RHS 1},
\end{align}
where the inequality is obtained by neglecting the terms associated with $( T_{e,\ell}/\alpha_\ell)^{M}$,  and the second equality follows from Lemma \ref{lmm: independent}. As $f_{0,F}$ is independent of the states affected by external arrivals to station $\ell>M\wedge J$, {the terms corresponding to jumps of external arrivals to station $\ell>M\wedge J$ are all zero.}

Furthermore, for each $\ell\in \J$, the term related to jumps of service completions at station $\ell$ is
\begin{align}
    \E_{s,\ell}\left[f_{0,F} \left(X^\uu+ \Delta_{s,\ell}\right) - f_{0,F} \left(X^\uu \right)  \right]   &= \E_{s,\ell}\left[ \left(\psi_{M }    +   \left(  \frac{T_{s,\ell}}{\mu_{\ell}^\uu}  \right)^{M }    \right)\left( \psi_{1}  +    \frac{T_{s,\ell}}{\mu_{\ell}^\uu}  \right) - \psi_{M}   \psi_{1}  \right] \notag \\
    & \geq     \E_{s,\ell}\left[  \psi_{M}  \frac{T_{s,\ell}}{\mu_\ell^\uu}   \right]  \geq \frac{\E_{s,\ell}\left[  \psi_{M}  \right] }{\lambda_\ell + 1}, \label{eq: base step RHS 2}
\end{align}
where the first inequality is obtained by neglecting the terms associated with $( T_{s,\ell}/\mu_\ell^\uu )^{M}$,  and the second inequality follows from Lemma \ref{lmm: independent} and {the fact that} ${\mu}^\uu_\ell =   \lambda_\ell + r^\ell \leq \lambda_\ell+ 1$ in Assumption~\ref{assmpt: multiscale}.

In summary, it follows from \eqref{eq: base step LHS}, \eqref{eq: base step RHS 1}, \eqref{eq: base step RHS 2} and BAR \eqref{eq:BAR} that
\begin{align*}
    \sum_{\ell =1}^{M \wedge J} \E_{e,\ell}\left[  \psi_{M}  \right] +\sum_{\ell=1}^J \frac{\lambda_\ell }{\lambda_\ell + 1} \E_{s,\ell}\left[  \psi_{M}  \right] \leq \left( 4J^2M + 2J \right) E_{1,0} .
\end{align*}
Therefore, Statement (\hyperlink{S4}{S4}) holds for $n=0$, all $\cc \in (0,1)$ and $1\leq k \leq M\wedge J$ with $$F_{k,0}\equiv  \left( 4J^2M + 2J \right) E_{1,0} \cdot  \max_{\ell\in \J}  \frac{\lambda_\ell + 1}{\lambda_\ell } < \infty.$$

\subsection{Proof of (S1)}\label{subsec: proof of S1}
Assuming that Statements (\hyperlink{S1}{S1})-(\hyperlink{S4}{S4}) hold for any pair in $S_{k,n}$ according to the induction hypotheses, we proceed to prove Statement (\hyperlink{S1}{S1}) for the pair $(k,n)$ by applying the test function $f_{k,n}$ defined by \eqref{eq:funcS1o} to the BAR \eqref{eq:BAR}.

To prove Statement (\hyperlink{S1}{S1}), we introduce two technical lemmas. Lemma \ref{lmm: eta} specifies the value of the constant $r_0$ as introduced in Theorem \ref{thm:Ckn}, and Lemma \ref{lmm: h simplify} shows that a linear combination of external arrival rates and service rates, with coefficients corresponding to the function $h_k$ in \eqref{eq:h}, can produce a desired order of $r^k$ as $r$ is small enough.
\begin{lemma}\label{lmm: eta}
Define
\begin{equation*}
    r_0 = \min_{\{1\leq k < j \leq J \mid  w_{jk} \neq 0\}} \left(\frac{   1-w_{kk}  }{J  w_{jk}}\right)^{\frac{1}{j-k}},
\end{equation*}
where the minimum is taken as $1$ if the set is empty.
Then, $r_0>0$, and for any $r\in (0,r_0)$ and any $k=1,\ldots,J$
\begin{equation} \label{eq: r0 condition}
    \left( 1-\ww{k}{k} \right)   \cc^k - \sum_{j=k+1}^J  \ww{j}{k}     \cc^j\ge \frac{1}{J}\left( 1-\ww{k}{k} \right)     \cc^k.
\end{equation}
\end{lemma}

\begin{proof}[Proof of Lemma \ref{lmm: eta}]
    According to Lemma \ref{lmm: w}, we have $w_{jk}\geq 0$ and $w_{kk}<1$ for any $j,k\in \J$. Then, $r < r_0$ implies that
    \begin{equation*}
        r < \left(\frac{   1-w_{kk}  }{J  w_{jk}}\right)^{\frac{1}{j-k}}\quad \text{ for any $1\leq k < j \leq J$ and $w_{jk}\neq 0$}.
    \end{equation*}
    Raising both sides of the inequality to the power of $j-k$ and then multiplying both sides by $w_{jk}r^k$, we have
    \begin{equation*}
        \ww{j}{k}   \cc^j < \frac{1}{J}\left( 1-\ww{k}{k} \right)     \cc^k \quad \text{ for any $1\leq k < j \leq J$ and $w_{jk}\neq 0$}.
    \end{equation*}
    Therefore, for any $k\in \J$,
    \begin{equation*}
        \sum_{j=k+1}^{J}\ww{j}{k}   \cc^j < \frac{J-k}{J}\left( 1-\ww{k}{k} \right) \cc^k \leq  \left( 1-\ww{k}{k} \right) \cc^k - \frac{1}{J}  \left( 1-\ww{k}{k} \right) \cc^k.
    \end{equation*}
    which implies \eqref{eq: r0 condition}.
\end{proof}
\begin{lemma}\label{lmm: h simplify}
    For any $k\in \J$ and $r\in (0,r_0)$, we have
    \begin{equation*}
        - \sum_{j=1}^k  \bb{k}{j}  \alpha_j +  {\mu}^\uu_k - \sum_{j=k}^{J} \ww{j}{k}  {\mu}^\uu_j \geq \frac{1}{J}\left( 1-\ww{k}{k} \right)     \cc^k.
    \end{equation*}
\end{lemma}
\begin{proof}[Proof of Lemma \ref{lmm: h simplify}]
    By utilizing the traffic equation \eqref{eq:traffic}, the definition of $u$ in \eqref{eq:u} and Lemma \ref{lmm: w}, we obtain
    \begin{align*}
        &  - \sum_{j=1}^k  \bb{k}{j}  \alpha_j + {\mu}^\uu_k - \sum_{j=k}^{J} \ww{j}{k}  {\mu}^\uu_j =  -   \sum_{j=1}^J  \bb{k}{j} \left( \lambda_{j}-\sum_{\ell =1}^J \lambda_\ell P_{\ell j} \right) + {\mu}^\uu_k  - \sum_{j=k}^{J} \ww{j}{k}  {\mu}^\uu_j \\
        &\quad=  -    \sum_{j=1}^J  \lambda_{j} \bb{k}{j} +  \sum_{j =1}^J \lambda_j \ww{j}{k}  +{\mu}^\uu_k  - \sum_{j=k}^{J} \ww{j}{k}  {\mu}^\uu_j =\left( {\mu}^\uu_k  - \lambda_k \right) - \sum_{j=k}^J  \ww{j}{k}  \left(  \mu^\uu_j -\lambda_j \right)  \\
        &\quad= \left( 1-\ww{k}{k} \right)   \cc^k - \sum_{j=k+1}^J  \ww{j}{k}     \cc^j\ge \frac{1}{J}\left( 1-\ww{k}{k} \right)     \cc^k,
    \end{align*} 
    where the inequality follows from Lemma \ref{lmm: eta}.
\end{proof}

We are now prepared to prove (\hyperlink{S1}{S1}) by applying the test function $f_{k,n}$, as defined in~\eqref{eq:funcS1o}, to BAR \eqref{eq:BAR}. The left-hand side of BAR becomes
\begin{align}
	&-\E_{\pi}\left[\A f_{k,n}\left( X^\uu \right)  \right]=-\E_{\pi}\left[\cc^{k(n-1)} \left(\b{k}^\T Z^{\uu} \right)^n \A h_{k}\left( R_{e}^\uu, R_{s}^\uu \right)  \right]\notag\\
	&\quad=\E_{\pi}\left[\cc^{k(n-1)} \left(\b{k}^\T Z^{\uu} \right)^n  \left( -\sum_{j=1}^k \bb{k}{j}  \alpha_j   +   {\mu}^\uu_k \I{Z_k^\uu>0}  -  \sum_{j=k}^{J} \ww{j}{k}  {\mu}^\uu_j \I{Z_j^\uu>0}  \right)     \right]\notag\\
	&\quad\geq 
    \E_{\pi}\left[\cc^{k(n-1)} \left(\b{k}^\T Z^{\uu} \right)^n  \left(  \left(  - \sum_{j=1}^k  \bb{k}{j}  \alpha_j +  {\mu}^\uu_k - \sum_{j=k}^{J} \ww{j}{k}  {\mu}^\uu_j \right)   -   {\mu}^\uu_k     \I{Z_k^\uu=0} \right)    \right]\notag \\
    &\quad \geq  \E_{\pi}\left[\cc^{k(n-1)} \left(\b{k}^\T Z^{\uu} \right)^n  \left(  \frac{1}{J}(1-w_{kk})r^k  -   {\mu}^\uu_k     \I{Z_k^\uu=0} \right)    \right],\label{eq:LHS}
\end{align}
where the last inequality follows from Lemma \ref{lmm: h simplify}.
The indicator part in \eqref{eq:LHS} can be bounded as follows:
\begin{align}
	&\E_{\pi}\left[\cc^{k(n-1)} \left(\b{k}^\T Z^{\uu} \right)^n  {\mu}^\uu_k   \I{Z_k^\uu=0}\right]= \cc^{n-k} {\mu}^\uu_k \E_{\pi}\left[ \left(\cc^{k-1} \sum_{j=1}^{k-1}  \bb{k}{j}  {Z}_j^{\uu} \right)^{n} \I{Z_k^\uu=0}  \right] \notag \\
    & \leq {\mu}^\uu_k  \E_{\pi}\left[  \left(\cc^{k -1}\sum_{j=1}^{k-1}  \bb{k}{j}  {Z}_j^{\uu} \right)^{n\vee k} \right]^{\frac{n}{n\vee k}} r^{n-\frac{kn}{n\vee k}} \leq \left( \lambda_k +1 \right) \E_{\pi}\left[  \left(\cc^{k -1}\sum_{j=1}^{k-1}  \bb{k}{j}  {Z}_j^{\uu} \right)^{n\vee k} \right]^{\frac{n}{n\vee k}}, \label{eq:indicator}
\end{align}
where the first inequality follows from Hölder's inequality, and the second inequality is obtained by $n - kn / (n \vee k) \geq 0$ and  ${\mu}^\uu_k =   \lambda_k + r^k \leq \lambda_k+ 1$ in Assumption~\ref{assmpt: multiscale}. According to Statement~(\hyperlink{S1}{S1}) of the induction hypotheses, the final term in \eqref{eq:indicator} can be further bounded by a constant $\Theta_1$, which is independent of $r\in (0,r_0)$.
As a result, we can conclude that the left-hand side of the BAR becomes
\begin{equation}\label{eq:LHS result}
    -\E_{\pi}\left[\A f_{k,n}\left(X^\uu\right)  \right] \ge \frac{1}{J} \left( 1-\ww{k}{k} \right)    \E_{\pi}\left[  \left( \cc^{k}\b{k}^\T {Z}^{\uu} \right)^n   \right]  - \Theta_1  .
\end{equation}

In the following derivation, we aim to demonstrate that the right-hand side of the BAR is uniformly bounded. 

For each $\ell \leq M\wedge J$, the term that corresponds to jumps of external arrivals to station $\ell$ is
\begin{align}
    &\E_{e,\ell}\left[f_{k,n}\left(X^\uu+ \Delta_{e,\ell}\right) - f_{k,n}\left(X^\uu \right)  \right]\notag\\
    &\quad=\cc^{k(n-1)} \E_{e,\ell}\left[ \frac{1}{n+1} \left(  \left(\b{k}^\T  {Z}^{\uu} + \bb{k}{\ell} \right)^{n+1}-  \left(\b{k}^\T   {Z}^{\uu}  \right) ^{n+1}  \right)- \left(\b{k}^\T  {Z}^{\uu} + \bb{k}{\ell} \right)^n   \bb{k}{\ell}  {T}_{e,\ell}   \right]\notag\\
    &\qquad +\cc^{k(n-1)} \E_{e,\ell}\left[\left(  \left(\b{k}^\T  {Z}^{\uu} + \bb{k}{\ell} \right)^n  -  \left(\b{k}^\T   {Z}^{\uu}   \right)^n   \right)  h_{k} \left( R_{e}^\uu, R_{s}^\uu \right)  \right],\label{eq:RHS1}
\end{align}
Following the mean value theorem, there exists a random variable $\xi_{\ell,1}\in (0,1)$, such that the first term in \eqref{eq:RHS1} is
\begin{align*}
    & \cc^{k(n-1)}\E_{e,\ell}\left[ \frac{1}{n+1} \left(  \left(\b{k}^\T  {Z}^{\uu} + \bb{k}{\ell} \right)^{n+1}-  \left(\b{k}^\T   {Z}^{\uu}  \right)^{n+1}  \right) - \left(\b{k}^\T  {Z}^{\uu} + \bb{k}{\ell} \right)^n   \bb{k}{\ell}  {T}_{e,\ell}   \right]\\
    &\quad =\cc^{k(n-1)} \E_{e,\ell}\left[\bb{k}{\ell}  \left(\b{k}^\T  {Z}^{\uu} + \xi_{\ell,1} \bb{k}{\ell} \right)^{n} -\bb{k}{\ell} \left(\b{k}^\T  {Z}^{\uu} + \bb{k}{\ell} \right)^n     \right]\leq 0,
\end{align*}
where the omission of $T_{e,\ell}$ is due to the use of Lemma \ref{lmm: independent}  and $\E[T_{e,\ell}]=1$.
Similarly, the second term in \eqref{eq:RHS1} can be written in terms of a random variable $\xi_{\ell,2}\in(0,1)$:
\begin{align}
    &\cc^{k(n-1)} \E_{e,\ell}\left[\left(  \left(\b{k}^\T  {Z}^{\uu} + \bb{k}{\ell} \right)^n  -  \left(\b{k}^\T   {Z}^{\uu}   \right)^n   \right)  h_{k} \left( R_{e}^\uu, R_{s}^\uu \right)  \right] \notag \\
    &\quad = n\bb{k}{\ell}  \E_{e,\ell}\left[ \cc^{k(n-1)} \left(\b{k}^\T  {Z}^{\uu} + \xi_{\ell,2}\bb{k}{\ell} \right)^{n-1}   h_{k}\left({R}_e^\uu, {R}_s^\uu\right)  \right] \notag \\
    &\quad \leq n\bb{k}{\ell}  \E_{e,\ell}\left[ \cc^{k(n-1)} \left(\b{k}^\T  {Z}^{\uu} +  \bb{k}{\ell} \right)^{n-1}  \abs{h_{k}\left({R}_e^\uu, {R}_s^\uu\right)}  \right]. \label{eq:RHS1-2}
\end{align}
According to Statement~(\hyperlink{S4}{S4}) of the induction hypotheses, the final term in \eqref{eq:RHS1-2} can be bounded by a constant  $\Theta_2$, which is independent of $r\in (0,r_0)$ and $\ell \leq M\wedge J$.
Therefore, we have
\begin{equation} \label{eq:RHS1 result}
    \E_{e,\ell}\left[ f_{k,n} \left(X^\uu+ \Delta_{e,\ell}\right) - f_{k,n} \left(X^\uu\right) \right]\leq \Theta_2.
\end{equation} 
As $f_{k,n}$ is independent of states affected by external arrivals to station $\ell>M\wedge J$, the terms corresponding to jumps of external arrivals to station $\ell>M\wedge J$ are all zero.

For each $\ell\in \J$, the term corresponding to jumps of service completions at station $\ell$ is
\begin{align}
    &\E_{s,\ell}\left[ f_{k,n} \left(X^\uu+ \Delta_{s,\ell}\right) - f_{k,n} \left(X^\uu\right) \right]\notag\\
    &=  {\cc^{k(n-1)}} \E_{s,\ell}\left[\frac{1}{n+1}\left(\left(\b{k}^\T {Z}^{\uu} + \Delta_Z \right)^{n+1} -   \left(\b{k}^\T   {Z}^{\uu}   \right)^{n+1}\right)   +     \left(\b{k}^\T   {Z}^{\uu} + \Delta_Z \right)^n    \left( \bb{k}{\ell} - \ww{\ell}{k} \right)   T_{s,\ell}     \right]\notag\\
    &\quad +\cc^{k(n-1)}  \E_{s,\ell}\left[ \left(  \left(\b{k}^\T   {Z}^{\uu} + \Delta_Z \right)^n  -  \left(\b{k}^\T   {Z}^{\uu}   \right)  ^n \right) h_{k}\left({R}_e^\uu, {R}_s^\uu\right)  \right] \label{eq:RHS2},
\end{align} 
where we use the fact that $$u_\ell - w_{\ell k} = \begin{cases}
    0 & \text{if $\ell<k$}\\
    1 - w_{kk} & \text{if $\ell=k$}\\
    -w_{\ell k} & \text{if $\ell>k$}
\end{cases}$$ and $\Delta_Z=- \b{k}^\T\left( \e{\ell} - \route{\ell} \right)$ is the increment upon the jump. According to Lemma \ref{lmm: independent} and \ref{lmm: w}, the increment $\Delta_Z$ satisfies
\begin{equation}\label{eq:Z properties}
    \abs{\Delta_Z}\leq J, \quad \E[\Delta_Z]=-\bb{k}{\ell}+\ww{\ell}{k}, \quad \text{and $\Delta_Z$ is independent of $Z^\uu$}.
\end{equation}
By the mean value theorem, there exist random variables $\xi_{\ell, 3}\in (0,1)$, $\xi_{\ell,4}\in (0,\xi_{\ell,3})$ and $\xi_{\ell, 5}\in (0,1)$, such that the first term in \eqref{eq:RHS2} becomes
\begin{align}
     & {\cc^{k(n-1)}} \E_{s,\ell}\left[\frac{1}{n+1}\left(\left(\b{k}^\T {Z}^{\uu} + \Delta_Z \right)^{n+1} -   \left(\b{k}^\T   {Z}^{\uu}   \right)^{n+1}\right)   +     \left(\b{k}^\T   {Z}^{\uu} + \Delta_Z \right)^n    \left( \bb{k}{\ell} - \ww{\ell}{k} \right)   T_{s,\ell}     \right] \notag \\
     &\quad ={\cc^{k(n-1)}} \E_{s,\ell}\left[\Delta_Z\left(\b{k}^\T {Z}^{\uu} + \xi_{\ell,3}\Delta_Z \right)^{n}    +  \left( \bb{k}{\ell} - \ww{\ell}{k} \right)    \left(\b{k}^\T   {Z}^{\uu} + \Delta_Z \right)^n        \right] \notag \\
     &\quad ={\cc^{k(n-1)}} \E_{s,\ell}\left[\Delta_Z\left(\left(\b{k}^\T {Z}^{\uu} + \xi_{\ell,3}\Delta_Z \right)^{n} - \left(\b{k}^\T {Z}^{\uu}  \right)^{n} \right)  \right] \notag \\
     & \qquad + {\cc^{k(n-1)}} \E_{s,\ell}\left[ \Delta_Z \left(\b{k}^\T {Z}^{\uu} \right)^{n} +  \left( \bb{k}{\ell} - \ww{\ell}{k} \right)    \left(\b{k}^\T   {Z}^{\uu} + \Delta_Z \right)^n        \right] \notag\\
     &\quad ={\cc^{k(n-1)}} \E_{s,\ell}\left[\Delta_Z\left(\left(\b{k}^\T {Z}^{\uu} + \xi_{\ell,3}\Delta_Z \right)^{n} - \left(\b{k}^\T {Z}^{\uu}  \right)^{n} \right)  \right] \notag\\
     & \qquad + {\cc^{k(n-1)}} \left( \bb{k}{\ell} - \ww{\ell}{k} \right)    \E_{s,\ell}\left[   \left(\b{k}^\T   {Z}^{\uu} + \Delta_Z \right)^n   -  \left(\b{k}^\T {Z}^{\uu} \right)^{n}      \right] \notag \\
     &\quad ={\cc^{k(n-1)}} \E_{s,\ell}\left[n\xi_{\ell,3}\Delta_Z^2 \left(\b{k}^\T {Z}^{\uu} + \xi_{\ell,4}\Delta_Z \right)^{n-1} +  \left( \bb{k}{\ell} - \ww{\ell}{k} \right) n \Delta_Z   \left(\b{k}^\T   {Z}^{\uu} + \xi_{\ell, 5}\Delta_Z \right)^{n-1}         \right] \notag\\
     &\quad \leq {\cc^{k(n-1)}} \E_{s,\ell}\left[nJ^2 \left(\b{k}^\T {Z}^{\uu} + J \right)^{n-1} +  nJ^2   \left(\b{k}^\T   {Z}^{\uu} + J \right)^{n-1}         \right], \label{eq:RHS2-1}
\end{align}
where the first equality follows from the mean value theorem, Lemma \ref{lmm: independent} and $\E[T_{s,\ell}]=1$, the third equality is derived from \eqref{eq:Z properties}, and the inequality follows from the fact that $x\leq |x|$ for any $x$, $\abs{\Delta_Z}\leq J$ in \eqref{eq:Z properties} and $\abs{u_\ell - w_{\ell k}}\leq J$. According to Statement~(\hyperlink{S2}{S2}) of the induction hypotheses, \eqref{eq:RHS2-1} can be bounded by a constant $\Theta_3$, which is a constant independent of $r\in (0,r_0)$ and $\ell \in \J$.
Similar to \eqref{eq:RHS1-2}, we can bound the second term of \eqref{eq:RHS2} as follows:
\begin{align*}
    &\cc^{k(n-1)}  \E_{s,\ell}\left[ \left(  \left(\b{k}^\T   {Z}^{\uu} + \Delta_Z \right)^n  - \left(\b{k}^\T   {Z}^{\uu}   \right)^n   \right)  h_{k}\left({R}_e^\uu, {R}_s^\uu\right)  \right]\leq \Theta_4,
\end{align*}
where $\Theta_4$ is a constant independent of $r\in (0,r_0)$ and $\ell \in \J$, according to Statement~(\hyperlink{S4}{S4}) of the induction hypotheses.
Therefore, we can conclude that for each $\ell \in \J$, the term corresponding to a jump of service completions satisfies
\begin{equation} \label{eq:RHS2 result}
    \E_{s,\ell}\left[ f_{k,n}\left(X^\uu+ \Delta_{s,\ell}\right) - f_{k,n}\left(X^\uu\right) \right]\leq \Theta_3 + \Theta_4.
\end{equation}

In summary, it follows from \eqref{eq:LHS result}, \eqref{eq:RHS1 result}, \eqref{eq:RHS2 result} and BAR \eqref{eq:BAR} that
\begin{align} \label{eq: result S1}
    \frac{1}{J} \left( 1-w_{kk} \right)    \E_{\pi}\left[ \left(\cc^{k}\b{k}^\T {Z}^{\uu} \right)^n   \right]  -  \Theta_1 \leq \Theta_2\sum_{\ell=1}^{M\wedge J}\alpha_\ell + (\Theta_3 + \Theta_4)\sum_{\ell=1}^{ J}\lambda_\ell.
\end{align}
According to $w_{kk}<1$ in Lemma \ref{lmm: w} and $u\in \R_+^J, u_k=1$ in \eqref{eq:u}, we conclude that $ \E_{\pi}[ (\cc^{k}  {Z}^{\uu}_k )^n   ] $ is uniformly bounded for $r\in(0,r_0)$ with
$$
C_{k,n} \equiv \frac{J}{1-w_{kk}}  \left( \Theta_1  + \Theta_2\sum_{\ell=1}^{M\wedge J}\alpha_\ell + (\Theta_3 + \Theta_4)\sum_{\ell=1}^{ J}\lambda_\ell \right)< \infty.
$$

\subsection{Proof of (S2)}\label{subsec: S2}
Based on the induction hypotheses, we assume that Statement (\hyperlink{S1}{S1}) holds for any pair in $S_{k,n}\cup \{(k,n)\}$, and Statements (\hyperlink{S2}{S2})-(\hyperlink{S4}{S4}) are satisfied for any pair in $ S_{k,n}$. We will now demonstrate the validity of Statement (\hyperlink{S2}{S2}) for pair $(k,n)$ by substituting $f_{k,n,D}$ as defined in \eqref{eq:funcS2o} into the BAR \eqref{eq:BAR}.

Plugging $f_{k,n,D}$ into the BAR, the left-hand side becomes
\begin{align} \label{eq:LHS-D}
    -\E_{\pi}\left[\A f_{k,n,D}\left(X^\uu\right) \right]&=   \E_{\pi}\left[ \left( \cc^k Z_k^\uu \right)^n \left( {M\wedge J}+\sum_{j=1}^J  \I{Z_{j}^{\uu}>0}\right)  \right]\leq 2JC_{k,n},
\end{align}
where the last inequality follows from Statement (\hyperlink{S1}{S1}) of the induction hypotheses.

For each $\ell \leq M\wedge J$, the term on the right-hand side of BAR that corresponds to jumps of external arrivals to station $\ell$ becomes
\begin{align}
    &\E_{e,\ell}\left[f_{k,n,D}\left(X^\uu+ \Delta_{e,\ell}\right) - f_{k,n,D}\left(X^\uu \right)  \right]  \notag  \\
    &  =\E_{e,\ell}\left[     \left( \cc^kZ_k^\uu + \cc^k\e{\ell}_k \right)^n\left(  \psi_{1} +    T_{e,\ell}/\alpha_\ell  \right) -  \left( \cc^kZ_k^\uu \right)^n \psi_{1}   \right]  \geq \E_{e,\ell}\left[ \left(\cc^{k} Z_k^\uu   \right)^n   \right]   /\alpha_\ell, \label{eq:RHS1-D}
\end{align}
where the inequality follows from Lemma \ref{lmm: independent}, $\E[T_{e,\ell}]=1$, and the fact that the increment $\cc^k\e{\ell}_k$ is nonnegative.
As $f_{k,n,D}$ is independent of states affected by external arrivals to station $\ell > M\wedge J$, the terms corresponding to jumps of external arrivals to station $\ell>M\wedge J$ are all zero.

For each $\ell \in \J$, the term on the right-hand side of BAR, corresponding to jumps of service completions at station $\ell$, can be expressed as
\begin{align}
    &\E_{s,\ell}\left[f_{k,n,D}\left(X^\uu+ \Delta_{s,\ell}\right) - f_{k,n,D}\left(X^\uu \right)  \right]  \notag\\
    &\quad =\E_{s,\ell}\left[     \left( \cc^{k}Z_k^\uu +\cc^{k}\Delta_Z \right)^n \left( \psi_{1}   +         {T_{s,\ell}}/{\mu_{\ell}^\uu} \right)  - \left( \cc^{k}Z_k^\uu \right)^n \psi_{1}  \right] \notag  \\
    &\quad = \E_{s,\ell}\left[     \left( \left( \cc^{k}Z_k^\uu +\cc^{k}\Delta_Z \right)^n -  \left( \cc^{k}Z_k^\uu \right)^n \right) \left( \psi_{1}   +       1/{\mu_{\ell}^\uu} \right)  + \left( \cc^{k}Z_k^\uu \right)^n  /{\mu_{\ell}^\uu}  \right]\label{eq:RHS2-D} 
\end{align}
where $\Delta_Z \equiv - \e{\ell}_k + \Route{\ell}{k}$ is the increment of queue length upon the jump, which satisfies the property that $\abs{\Delta_Z}\leq 1$. The final equality holds due to Lemma \ref{lmm: independent}.
The first term in~\eqref{eq:RHS2-D} is of order $n-1$ and consequently uniformly bounded. This fact can be proved by applying the mean value theorem. Specifically, there exists a random variable $\xi_{\ell}\in (0,1)$ such that the first term in \eqref{eq:RHS2-D} is given by
\begin{align}
    &\E_{s,\ell}\left[     \left( \left( \cc^{k}Z_k^\uu +\cc^{k}\Delta_Z \right)^n -  \left( \cc^{k}Z_k^\uu \right)^n \right) \left( \psi_{1}   +       1/{\mu_{\ell}^\uu} \right)  \right] \notag \\
    &\quad = \E_{s,\ell}\left[      n \cc^{k}  \Delta_Z  \left( \cc^{k}Z_k^\uu +\xi_{\ell}\cc^{k}\Delta_Z \right)^{n-1} \left( \psi_{1}   +       1/{\mu_{\ell}^\uu} \right)    \right] \notag \\
    &\quad \geq - \E_{s,\ell}\left[      n \cc^{k}    \left( \cc^{k}Z_k^\uu + 1\right)^{n-1} \left( \psi_{1}   +       1/{\lambda_{\ell}} \right)    \right] , \label{eq:RHS2-D-1}
\end{align}
where the inequality follows from the fact that $x\geq -\abs{x}$ for any $x$, $\abs{\Delta_Z}\leq 1$, $\cc,\xi_{\ell}\in (0,1)$ and $\mu_\ell^\uu>\lambda_\ell$.  According to Statements (\hyperlink{S2}{S2}) and (\hyperlink{S4}{S4}) of the induction hypotheses, the final term in \eqref{eq:RHS2-D-1} can be further lower bounded by a constant  $-\Theta_5$, which is independent of $r\in (0,r_0)$ and $\ell\in \J$.
Thus, \eqref{eq:RHS2-D} becomes
\begin{align} \label{eq:RHS2-D result}
    \E_{s,\ell}\left[f_{k,n,D}\left(X^\uu+ \Delta_{s,\ell}\right) - f_{k,n,D}\left(X^\uu \right)  \right] &  \geq   \E_{s,\ell}\left[  \left( \cc^{k}Z_k^\uu \right)^n  /\mu_\ell^\uu\right] -\Theta_5\notag \\
    &\geq   \E_{s,\ell}\left[  \left( \cc^{k}Z_k^\uu \right)^n  /(\lambda_\ell + 1)\right] -\Theta_5,
\end{align}
where the last inequality follows from  the fact that  ${\mu}^\uu_\ell =   \lambda_\ell + r^\ell \leq \lambda_\ell+ 1$ in Assumption~\ref{assmpt: multiscale}.

In summary, it follows from \eqref{eq:LHS-D}, \eqref{eq:RHS1-D}, \eqref{eq:RHS2-D result} and BAR \eqref{eq:BAR} that
\begin{align*}
    2JC_{k,n} \geq \sum_{\ell =1}^{M\wedge J}  \E_{e,\ell}\left[ \left(\cc^{k} Z_k^\uu   \right)^n   \right]   + \sum_{\ell=1}^J \frac{\lambda_\ell}{\lambda_\ell+1}   \E_{s,\ell}\left[ \left( \cc^{k}Z_k^\uu \right)^n \right]-\Theta_5\sum_{\ell=1}^J \lambda_\ell .
\end{align*}
Therefore, Statement (\hyperlink{S2}{S2}) holds for pair $(n,k)$ and all $\cc \in (0,r_0)$ with
$$D_{k,n}\equiv  \left( 2JC_{k,n}  + \Theta_5\sum_{\ell=1}^J \lambda_\ell  \right)   \max_{\ell\in \J}  \frac{\lambda_\ell + 1}{\lambda_\ell } < \infty.$$

\subsection{Proof of (S3)}
Building upon the induction hypotheses, we assume that Statements (\hyperlink{S1}{S1}) and (\hyperlink{S2}{S2}) hold for any pair in $S_{k,n}\cup \{(k,n)\}$, while Statements (\hyperlink{S3}{S3}) and (\hyperlink{S4}{S4}) are satisfied for any pair in $ S_{k,n}$. In this section, we aim to prove Statement (\hyperlink{S3}{S3}) for the pair $(k,n)$.  Note that when $n=M$, Statement (\hyperlink{S3}{S3}) holds trivially due to Statement (\hyperlink{S1}{S1}).  Hence, in the subsequent analysis, we will solely focus on the cases when $1\leq n<M$.

Plugging $f_{k,n,E}$ \eqref{eq:funcS3o} into BAR \eqref{eq:BAR}, the left-hand side becomes
\begin{align} 
    &-\E_{\pi}\left[\A f_{k,n,E}\left(X^\uu\right) \right] \notag \\
    &\quad = (M-n+1) \E_{\pi}\left[ \left(\cc^{k}  Z_k^\uu \right)^n    \left(  \sum_{j=1}^{M\wedge J} \left(  R_{e,j}^\uu \right)^{M-n} + \sum_{j=1}^J  \left( R_{s,j}^\uu \right)^{M-n} \I{Z_{j}^{\uu}>0}  \right)\right] \label{eq:LHS-E}
\end{align}
For each $\ell \leq M\wedge J$, the term on the right-hand side of BAR that corresponds to jumps of external arrivals to station $\ell$ can be expressed as follows:
\begin{align}
    &\E_{e,\ell}\left[f_{k,n,E}\left(X^\uu+ \Delta_{e,\ell}\right) - f_{k,n,E}\left(X^\uu \right)  \right]  \notag\\
    &  \quad=\E_{e,\ell}\left[     \left(\cc^{k}  Z_k^\uu + \cc^{k}\e{\ell}_k \right)^n\left( \psi_{M-n+1} +  \left( 
    {T_{e,\ell}}/{\alpha_\ell} \right)^{M-n+1}\right)  -  \left(\cc^{k}  Z_k^\uu \right)^n  \psi_{M-n+1}   \right] \notag \\
    &  \quad = \E_{e,\ell}\left[    \left(  \left(\cc^{k}  Z_k^\uu + \cc^{k}\e{\ell}_k \right)^n - \left(\cc^{k}  Z_k^\uu \right)^n \right) \left( \psi_{M-n+1} + {\E\left[T_{e,\ell}^{M-n+1}   \right]}/{\alpha_\ell^{M-n+1}} \right)  \right., \notag\\
    &  \qquad   \left.   +    \left(\cc^{k}  Z_k^\uu  \right)^n  {\E\left[T_{e,\ell}^{M-n+1}   \right]}/{\alpha_\ell^{M-n+1}} \right], \label{eq:LHS-E-1}
\end{align}
where the last equality holds due to Lemma \ref{lmm: independent}. According to the condition \eqref{eq: moment condtion}, the mean value theorem and Statements (\hyperlink{S2}{S2}) and (\hyperlink{S4}{S4}) of the induction hypotheses, the final term in \eqref{eq:LHS-E-1} can be further bounded by a constant $\Theta_6$, which is independent of $r\in (0,r_0)$ and $\ell \leq M \wedge J$.
Consequently, we can conclude that
\begin{align} \label{eq:RHS1-E}
    \E_{e,\ell}\left[f_{k,n,E}\left(X^\uu+ \Delta_{e,\ell}\right) - f_{k,n,E}\left(X^\uu \right)  \right] \leq  \Theta_6,
\end{align}
As $f_{k,n,E}$ is independent of states affected by external arrivals to station $\ell > M\wedge J$, the terms corresponding to jumps of external arrivals to station $\ell>M\wedge J$ are all zero.

Furthermore, for each $\ell\in \J$, the term on the right-hand side of BAR corresponding to the jumps of the service completion at station $\ell$ can be expressed as
\begin{align}
    &\E_{s,\ell}\left[f_{k,n,E} \left(X^\uu+ \Delta_{s,\ell}\right) - f_{k,n,E} \left(X^\uu \right)  \right]  \notag\\
    &\quad =\E_{s,\ell}\left[     \left(\cc^{k}  Z_k^\uu + \cc^{k}\Delta_Z \right)^n\left( \psi_{M-n+1} +  \left(  {T_{s,\ell}}/{\mu_\ell^\uu} \right)^{M-n+1}\right)  -  \left(\cc^{k}  Z_k^\uu \right)^n  \psi_{M-n+1}   \right] \notag \\
    & \quad = \E_{s,\ell}\left[    \left(  \left(\cc^{k}  Z_k^\uu + \cc^{k}\Delta_Z \right)^n - \left(\cc^{k}  Z_k^\uu \right)^n \right) \left( \psi_{M-n+1} + {\E\left[T_{s,\ell}^{M-n+1}   \right]}/{\left(\mu_\ell^\uu\right)^{M-n+1}} \right) \right.\notag \\
    & \qquad \left. + \left(\cc^{k}  Z_k^\uu  \right)^n  {\E\left[T_{s,\ell}^{M-n+1}   \right]}/{\left(\mu_\ell^\uu\right)^{M-n+1}}\right],  \label{eq:LHS2-E-1}
\end{align} 
where  $\Delta_Z\equiv- \e{\ell}_k +\Route{\ell}{k}$ is the increment upon a jump satisfying $\abs{\Delta_Z}\leq 1$. 
Similar to \eqref{eq:LHS-E-1}, according to the condition \eqref{eq: moment condtion}, the mean value theorem, Statements (\hyperlink{S2}{S2}) and (\hyperlink{S4}{S4}) of the induction hypotheses and the fact that $\mu^\uu_\ell> \lambda_\ell$ for all $\ell\in \J$, the final term in \eqref{eq:LHS2-E-1} can be further bounded by a constant $\Theta_7$, which is independent of $r\in (0,r_0)$ and $\ell\in \J$.
Therefore, we can conclude that 
\begin{align} \label{eq:RHS2-E result}
    \E_{s,\ell}\left[f_{k,n,E} \left(X^\uu+ \Delta_{s,\ell}\right) - f_{k,n,E} \left(X^\uu \right)  \right] \leq  \Theta_7.
\end{align}

In summary, it follows from \eqref{eq:LHS-E}, \eqref{eq:RHS1-E}, \eqref{eq:RHS2-E result} and BAR \eqref{eq:BAR} that
\begin{align}
    & \E_{\pi}\left[ \left(\cc^{k}  Z_k^\uu \right)^n    \left(  \sum_{j=1}^{M\wedge J} \left(  R_{e,j}^\uu \right)^{M-n} + \sum_{j=1}^J  \left( R_{s,j}^\uu \right)^{M-n} \I{Z_{j}^{\uu}>0}  \right)\right]  \leq \frac{\Theta_6\sum_{\ell=1}^{M\wedge J}\alpha_\ell +\Theta_7\sum_{\ell=1}^{J}\lambda_\ell}{(M-n+1)}. \label{eq: S3-1}
\end{align}
As $R_{s,j}^\uu\overset{d}{=}T_{s,j} /\mu_j^\uu$ when the station $j$ is idle for any $j\in \J$, we have
\begin{align}
    &\E_{\pi}\left[ \left( \cc^{k} Z_k^\uu \right)^n \sum_{j=1}^J  \left( R_{s,j}^\uu \right)^{M-n} \I{Z_{j}^{\uu}=0}   \right]=\E_{\pi}\left[  \left(  \cc^{k}Z_k^\uu \right)^n \sum_{j=1}^J  \left(  \frac{T_{s,j}}{\mu_j^\uu}  \right)^{M-n}  \I{Z_{j}^{\uu}=0}   \right]  \notag \\
    &\quad\leq \E_{\pi}\left[ \left(\cc^{k}  Z_k^\uu \right)^n \right] \sum_{j=1}^J \left(\mu_j^\uu\right)^{-(M-n)} \E \left[  T_{s,j} ^{M-n}   \right] \leq   C_{k,n} \sum_{j=1}^J {\lambda_j}^{-(M-n)}\E \left[   T_{s,j} ^{M-n}   \right], \label{eq: S3-2}
\end{align}
where the last inequality holds due to Statement (\hyperlink{S2}{S2}) and $\mu_j^\uu> \lambda_j$ for all $j\in \J$.
Hence, by combining \eqref{eq: S3-1} and \eqref{eq: S3-2}, we can conclude that for any $r\in (0,r_0)$
\begin{align*}
    \E_{\pi}\left[  \left( \cc^{k}Z_k^\uu \right)^n \psi_{M-n}  \right]&=\E_{\pi}\left[  \left( \cc^{k}Z_k^\uu \right)^n \left(  \sum_{j=1}^{M\wedge J}  \left(  R_{e,j}^\uu  \right)^{M-n} +\sum_{j=1}^J   \left( R_{s,j}^\uu \right)^{M-n} \right)  \right] \\
    &  \leq \frac{\Theta_6\sum_{\ell=1}^{M\wedge J}\alpha_\ell +\Theta_7\sum_{\ell=1}^{J}\lambda_\ell}{(M-n+1)}+  C_{k,n} \sum_{j=1}^J {\lambda_j}^{-(M-n)}\E \left[   T_{s,j} ^{M-n}   \right] \equiv E_{k,n}< \infty.
\end{align*}
Therefore, Statement (\hyperlink{S3}{S3}) holds for pair $(n,k)$ and all $\cc \in (0,r_0)$ with $E_{k,n}$ as given above.

\subsection{Proof of (S4)}\label{subsec: S4}
Given the induction hypotheses, we assume that Statements (\hyperlink{S1}{S1})-(\hyperlink{S3}{S3}) hold for all pairs in the set $S_{k,n}\cup \{(k,n)\}$, and Statement (\hyperlink{S4}{S4}) is satisfied for all pairs in $S_{k,n}$. We will now prove Statement (\hyperlink{S4}{S4}) for the pair $(k,n)$. However, if $n=M$, Statement (\hyperlink{S4}{S4}) is trivially satisfied due to Statement (\hyperlink{S2}{S2}). Therefore, in the following analysis, we will focus on the cases where $1\leq n<M$.

Applying $f_{k,n,F}$ \eqref{eq:funcS4o} into BAR \eqref{eq:BAR}, the left-hand side becomes
\begin{align}
    &-\E_{\pi}\left[\A f_{k,n,F}\left(X^\uu\right) \right]= \E_{\pi}\left[ \left( \cc^{k}  Z_k^\uu \right)^n  \left(  {M\wedge J}  +\sum_{j=1}^J \I{Z_{j}^{\uu}>0} \right) \psi_{M-n}   \right] \notag\\
    &\qquad +(M-n) \E_{\pi}\left[ \left( \cc^{k}  Z_k^\uu \right)^n   \left( \sum_{j=1}^{M\wedge J} \left( R_{e,j}^\uu \right)^{M-n-1} +\sum_{j=1}^J \left( R_{s,j}^\uu \right)^{M-n-1} \I{Z_{j}^{\uu}>0} \right)   \psi_{1}  \right] \notag\\
    &\quad\leq    2J\E_{\pi}\left[ \left( \cc^{k}  Z_k^\uu \right)^n \psi_{M-n}   \right] + (M-n) \E_{\pi}\left[ \left(  \cc^{k}  Z_k^\uu \right)^n\psi_{M-n-1}   \psi_{1}  \right] \leq \left( 4J^2 M+ 2J \right) E_{k,n}, \label{eq: LHS}
\end{align}
where the last inequality holds due to the property stated in \eqref{eq: cross term bound} and Statement (\hyperlink{S3}{S3}) of the induction hypotheses.

For each $\ell \leq M\wedge J$, the term on the right-hand side of BAR that corresponds to jumps of external arrivals to station $\ell$ becomes
\begin{align}
    &\E_{e,\ell}\left[f_{k,n,F}\left(X^\uu+ \Delta_{e,\ell}\right) - f_{k,n,F}\left(X^\uu \right)  \right]\notag\\
    &\quad= \E_{e,\ell} \left[  \left(\cc^{k} Z_k^\uu + \cc^{k}\e{\ell}_k \right)^n\left(\psi_{M-n}   +   \left(   \frac{T_{e,\ell}}{\alpha_{\ell}} \right)^{M-n}    \right)\left( \psi_{1}  +     \frac{T_{e,\ell}}{\alpha_{\ell}}   \right) - \left(\cc^{k} Z_k^\uu  \right)^n \psi_{M-n} \psi_{1}  \right] \notag\\
    &\quad\geq \E_{e,\ell} \left[  \left( \left(\cc^{k} Z_k^\uu + \cc^{k}\e{\ell}_k \right)^n - \left(\cc^{k} Z_k^\uu \right)^n \right) \psi_{M-n}   \left( \psi_{1}  +    1/{\alpha_{\ell}}  \right) + \left(\cc^{k} Z_k^\uu  \right)^n \psi_{M-n} / \alpha_\ell  \right] \notag\\
    &\quad\geq \E_{e,\ell} \left[  \left(\cc^{k} Z_k^\uu  \right)^n \psi_{M-n}  \right] / \alpha_\ell  \label{eq:RHS1-I}
\end{align}
where the first inequality is obtained by neglecting the terms associated with $\left( T_{e,\ell}/\alpha_\ell \right)^{M-n}$ and using Lemma \ref{lmm: independent} and $\E[T_{e,\ell}]=1$,  and the second inequality is valid as $\cc^k \e{\ell}_k $ is nonnegative. As $f_{k,n,F}$ is independent of states affected by external arrivals to station $\ell > M\wedge J$, the terms corresponding to jumps of external arrivals to station $\ell>M\wedge J$ are all zero.

Furthermore, for each $\ell\in \J$, the term on the right-hand side of BAR, corresponding to jumps of service completions at station $\ell$, becomes
\begin{align}
    & \E_{s,\ell}\left[f_{k,n,F}\left(X^\uu+ \Delta_{s,\ell}\right) - f_{k,n,F}\left(X^\uu \right)  \right]\notag\\
    &\quad = \E_{s,\ell}\left[  \left(\cc^{k} Z_k^\uu + \cc^{k}\Delta_Z \right)^n\left(\psi_{M-n}   +   \left(   \frac{T_{s,\ell}}{\mu_{\ell}^\uu} \right)^{M-n}    \right)\left( \psi_{1}  +     \frac{T_{s,\ell}}{\mu_{\ell}^\uu}   \right) - \left(\cc^{k} Z_k^\uu  \right)^n \psi_{M-n} \psi_{1}  \right] \notag\\
    &\quad\geq \E_{s,\ell}\left[  \left( \left(\cc^{k} Z_k^\uu + \cc^{k}\Delta_Z \right)^n - \left(\cc^{k} Z_k^\uu \right)^n \right) \psi_{M-n}   \left( \psi_{1}  +   {1}/{ {\mu_{\ell}^\uu}}  \right) + \left(\cc^{k} Z_k^\uu  \right)^n {\psi_{M-n}  }/{\mu_{\ell}^\uu} \right] \label{eq:RHS2-I} 
\end{align}
where $\Delta_Z\equiv - \e{\ell}_k + \Route{\ell}{k}$ is the increment upon the jump satisfying $\abs{\Delta_Z}\leq 1$. The inequality is derived by disregarding the terms related to $( T_{s,\ell}/\mu_\ell^\uu )^{M-n}$ and utilizing $Z_k^\uu + \Delta_Z\geq 0$, Lemma \ref{lmm: independent} and $\E[T_{s,\ell}]=1$.
We show that the first term in \eqref{eq:RHS2-I} is of order $n-1$ in $\cc^k Z_k^\uu$ and order $M-n+1$ in the function $\psi$ using the mean value theorem. In particular, there exists a random variable $\xi_{\ell}\in (0,1)$ such that the first term in \eqref{eq:RHS2-I} is given by
\begin{align}
    &\E_{s,\ell}\left[     \left( \left( \cc^{k}Z_k^\uu +\cc^{k}\Delta_Z \right)^n -  \left( \cc^{k}Z_k^\uu \right)^n \right)  \psi_{M-n}   \left( \psi_{1}  +    {1}/{ {\mu_{\ell}^\uu}}  \right)  \right] \notag \\
    &\quad = \E_{s,\ell}\left[      n \cc^{k}  \Delta_Z  \left( \cc^{k}Z_k^\uu +\xi_{\ell}\cc^{k}\Delta_Z \right)^{n-1}  \left(  \psi_{M-n}  \psi_{1}  +    { \psi_{M-n}  }/{ {\mu_{\ell}^\uu}}  \right)     \right] \notag \\
    &\quad \geq - \E_{s,\ell}\left[      n \cc^{k}    \left( \cc^{k}Z_k^\uu + 1\right)^{n-1} \left( 4J^2\psi_{M-n+1}   +       \psi_{M-n}/{{\mu_{\ell}^\uu}} \right)    \right], \label{eq:RHS2-I-1}
\end{align}
where the inequality follows from the fact that $x\geq -\abs{x}$ for any $x$, $\abs{\Delta_Z}\leq 1$ and $\cc,\xi_{\ell}\in (0,1)$. According to Statement (\hyperlink{S4}{S4}) of the induction hypotheses, the final term in \eqref{eq:RHS2-I-1} can be further lower bounded by a constant $-\Theta_8$, which is independent of $r \in (0, r_0)$ and $\ell \in \J$.
Therefore, we have
\begin{align} 
    \E_{s,\ell}\left[f_{k,n,F} \left(X^\uu+ \Delta_{s,\ell}\right) - f_{k,n,F} \left(X^\uu \right)  \right] & \geq   \E_{s,\ell}\left[ \left(  \cc^{k}  Z_k^\uu  \right)^n \psi_{M-n}/{\mu_{\ell}^\uu} \right]   - \Theta_8 \notag \\
    & \geq   \E_{s,\ell}\left[ \left(  \cc^{k}  Z_k^\uu  \right)^n \psi_{M-n} \right]  / (\lambda_\ell + 1)  - \Theta_8 ,\label{eq:RHS2-I result} 
\end{align}
where the last inequality follows from  the fact that  ${\mu}^\uu_\ell =   \lambda_\ell + r^\ell \leq \lambda_\ell+ 1$ in Assumption~\ref{assmpt: multiscale}.

In summary, it follows from \eqref{eq: LHS}, \eqref{eq:RHS1-I}, \eqref{eq:RHS2-I result} and BAR \eqref{eq:BAR} that
\begin{align*}
    & \sum_{\ell =1}^{M \wedge J}\E_{e,\ell}\left[   \left( \cc^kZ_k^\uu   \right)^n \psi_{M-n} \right] +\sum_{\ell=1}^J     \frac{\lambda_\ell}{\lambda_\ell+1}     \E_{s,\ell}\left[  \left( \cc^k Z_k^\uu  \right)^n \psi_{M-n} \right] \leq \left( 4J^2 M + 2J \right) E_{k,n} + \Theta_8\sum_{\ell=1}^J \lambda_\ell.
\end{align*}
Therefore, Statement (\hyperlink{S4}{S4}) holds for pair $(n,k)$ and all $\cc \in (0,r_0)$ with
$$
F_{k,n} \equiv \left( \left( 4J^2 M + 2J \right) E_{k,n} + \Theta_8\sum_{\ell=1}^J \lambda_\ell \right) \max_{\ell\in \J} \frac{\lambda_\ell+1}{\lambda_\ell}< \infty.
$$

\subsection{Truncated Test Functions}\label{sec: truncation}
To make the proof of Theorem \ref{thm:Ckn} more rigorous, we need to replace the unbounded test functions, as given by \eqref{eq:funcS1o} and \eqref{eq:funcS2o}-\eqref{eq:funcS4o}, with their truncated counterparts in the proof of Statements (\hyperlink{S1}{S1})-(\hyperlink{S4}{S4}). Initially, we will introduce the truncated test functions. Subsequently, we will demonstrate in Lemma~\ref{lmm: g properties} that these truncated test functions are ``similar" to the original unbounded test functions, in the sense that they also satisfy the desirable properties.   This allows us to assert that the derivations in Sections \ref{subsec: bounds for residual times}-\ref{subsec: S4} remain valid even when applied to the truncated test functions.

\paragraph{Truncated test functions.}
To handle the residual times in the test function, we truncate them by setting an upper bound $\kappa \in (0, \infty)$ following \citet{BravDaiMiya2017}, which we call the ``hard truncation". The truncated versions of functions $\psi_n(\cdot)$ in \eqref{eq: function psi} and $h_k(\cdot)$ in \eqref{eq:h}, indexed by $\kappa>0$, are defined as follows:
\begin{align*}
    \psi_{n}^\fk\left(r_e, r_s \right) &= \sum_{j=1}^{M\wedge J} \left(  r_{e,j} \wedge \kappa \right)^{n} +  \sum_{j=1}^J  \left( r_{s,j} \wedge \kappa \right)^{n},\\
	h^\fk_{k}\left({r}_e, {r}_s \right) &=  -  \sum_{j=1}^k  \bb{k}{j}  \alpha_j\left( {r}_{e,j} \wedge \kappa \right) +    {\mu}^\uu_k \left( {r}_{s,k}\wedge \kappa \right)  -\sum_{j=k}^{J} w_{jk} {\mu}^\uu_j \left( r_{s,j} \wedge \kappa \right),
\end{align*}
with for any $r_e,r_s\geq 0$, $\psi_{n}^\fk(r_e, r_s) \leq \psi_{n}(r_e, r_s)$ and $\psi_{n}^\fk(r_e, r_s) \rightarrow \psi_{n}(r_e, r_s)$ as $\kappa$ goes to infinity. We can check that for all $\kappa>0$, both functions $\psi_n^\fk$ and $h_k^\fk $ belong to $\D$. 

To truncate the queue length, we employ a ``soft truncation" incorporating an exponentially decaying function. 
Specifically, the soft-truncated version of the polynomial functions and the corresponding integral functions are given by:
\begin{equation*}
g^\fk_{n}(z) = z^n \exp(-z/\kappa), \qquad G_{n}^\fk\left(z\right) = \int_0^z  g^\fk_{n}(y) dy=\int_0^z y^n \exp(-y/\kappa) dy,
\end{equation*}
with for any $z\geq 0$, $g^\fk_{n}(z)\leq z^n$ and $g^\fk_{n}(z) \to z^n$ as $\kappa$ goes to infinity. 
We can check that for all $\kappa>0$, both functions $g_n^\fk$ and $G_n^\fk $ belong to $\D$. 

In summary, the truncated test functions, corresponding to \eqref{eq:funcS1o} and \eqref{eq:funcS2o}-\eqref{eq:funcS4o}, are as follows: for $x = (z, r_e, r_s) \in \mathbb{S}$,
\begin{align*}
	f^\fk_{k,n}\left(x\right) &= \cc^{k(n-1)}  G^\fk_{n}\left(\b{k}^\T {z}\right) + \cc^{k(n-1)}  g^\fk_{n}\left(\b{k}^\T z\right) h^\fk_{k}\left(r_e, r_s \right),\\
	f^{\fk}_{k,n,D}\left(x\right) &= \cc^{kn} g^\fk_{n}\left(z_k\right) \psi_{1}^\fk\left(r_e,r_s\right),\\
	f^\fk_{k,n,E}\left(x\right) &=  \cc^{kn} g^\fk_{n}\left(z_k\right)   \psi_{M-n+1}^\fk\left(r_e,r_s\right),\\
	f^{\fk}_{k,n,F}\left(x\right) &= \cc^{kn} g^\fk_{n}\left(z_k\right) \psi_{M-n}^\fk \left(r_e,r_s\right) \psi_{1}^\fk\left(r_e,r_s\right).
\end{align*}

The primary reason for using soft-truncated functions is to maintain smoothness, so that the mean value theorem in Sections \ref{subsec: bounds for residual times}-\ref{subsec: S4}, can still be applied. In the following lemma, we show that the truncated functions $g_n^\fk$ and $G_n^\fk$ indeed have similar properties as the polynomial terms $z^n$ and $z^{n+1}/(n+1)$ in the unbounded test functions.  

\begin{lemma}\label{lmm: g properties} For all $z\in\R_+$, the following properties are true:
	\begin{enumerate}[(i)]
		\item  
         $g_{n}^\fk(z) \leq \left( \kappa n/e \right)^n $.
		\item $\abs{g_{n}^\fk(z+c) - g_{n}^\fk(z)}\leq (n+1)|c|(z+|c|)^{n-1} $, for any real number $c\geq -z$.
		\item $G_{n}^\fk(z+c) - G_{n}^\fk(z)\leq c(z+ c)^n\exp\left( -z / \kappa \right) $, for any real number $c\geq -z$.
	\end{enumerate}
\end{lemma}
\begin{proof}[Proof of Lemma \ref{lmm: g properties}]
    Property (i) follows from the fact that 
    $x^n\exp(-x)\le (n/e)^n$ for all $x\ge 0$. For property (ii), we can check that
    \begin{align*}
        &\abs{g_{n}^\fk(z+c) - g_{n}^\fk(z)} = \abs{c  \dot{g}_{n}^\fk(z+\theta c) } = \abs{c \left( n(z+\theta c)^{n-1}  - \kappa^{-1}(z+\theta c)^n \right) \exp(-(z+\theta c)/\kappa) }\\
        &\quad \leq n\abs{c}  (z+\theta c)^{n-1} + \abs{c}  (z+\theta c)^{n-1} = (n+1)\abs{c}  (z+\theta c)^{n-1}  \leq (n+1) \abs{c} (z+\abs{c})^{n-1} ,   
    \end{align*} 
    where $\dot{g}_n^\fk$ is the derivative of $g_n^\fk$, and $\theta$ is a constant in $(0,1)$ by the mean value theorem. The first inequality holds as $x\exp(-x)\leq 1$ for $x=(z+\theta c)/\kappa$. For property (iii), we can check that
    \begin{align*}
        G_{n}^\fk(z+c) - G_{n}^\fk(z) = c  g_{n}^\fk(z+\bar{\theta} c) = c(z+\bar{\theta} c)^n\exp\left( -(z+\bar{\theta} c) / \kappa \right) \leq c(z+ c)^n\exp\left( -z / \kappa \right),
    \end{align*}
    where the constant $\bar{\theta}\in (0,1)$ is derived by the mean value theorem.
\end{proof}

\paragraph{Applying truncated test functions to BAR.} 
We next describe the necessary adjustments in the proofs of Sections \ref{subsec: bounds for residual times}-\ref{subsec: S4} when replacing unbounded test functions with truncated versions. For illustrative purposes, we focus on the proof of Statement (\hyperlink{S1}{S1}) in Section \ref{subsec: proof of S1}, since the modifications for Statements (\hyperlink{S2}{S2})-(\hyperlink{S4}{S4}) are less complex and follow a similar rationale. A comparative analysis with Section \ref{subsec: proof of S1} is provided below.

Applying $f_{k,n}^{(\kappa)}$ to BAR \eqref{eq:BAR}, the left-hand side becomes
\begin{align}
	&-\E_{\pi}\left[\A f_{k,n}^{(\kappa)}\left( X^\uu \right)  \right]=-\E_{\pi}\left[\cc^{k(n-1)} g_n^\fk\left(\b{k}^\T Z^{\uu} \right) \A h_{k}^\fk\left( R_{e}^\uu, R_{s}^\uu \right)  \right]\notag\\
	&\quad=\E_{\pi}\left[\cc^{k(n-1)} g_n^\fk\left(\b{k}^\T Z^{\uu} \right)  \left( -\sum_{j=1}^k \bb{k}{j}  \alpha_j \I{R_{e,j}^\uu\leq \kappa}   +   {\mu}^\uu_k \I{Z_k^\uu>0, R_{s,k}^\uu \leq \kappa} \right.\right. \notag \\
    & \qquad\qquad\qquad\qquad\qquad\qquad\qquad \left.\left. -  \sum_{j=k}^{J} \ww{j}{k}  {\mu}^\uu_j \I{Z_j^\uu>0, R_{s,j}^\uu\leq \kappa}  \right)     \right]\notag\\
	&\quad\geq 
    \E_{\pi}\left[\cc^{k(n-1)} g_n^\fk\left(\b{k}^\T Z^{\uu} \right)  \left(  \left(  - \sum_{j=1}^k  \bb{k}{j}  \alpha_j +  {\mu}^\uu_k - \sum_{j=k}^{J} \ww{j}{k}  {\mu}^\uu_j \right)  \right.\right. \notag \\
    & \qquad\qquad\qquad\qquad\qquad\qquad\qquad \left.\left.  -   {\mu}^\uu_k    \left( 1- \I{Z_k^\uu>0, R_{s,k}^\uu \leq \kappa} \right) \right)    \right]\notag \\
    &\quad \geq  \E_{\pi}\left[\cc^{k(n-1)} g_n^\fk\left(\b{k}^\T Z^{\uu} \right)  \left(  \frac{1}{J}(1-w_{kk})r^k  -   {\mu}^\uu_k     \I{Z_k^\uu=0} -   {\mu}^\uu_k     \I{R_{s,k}^\uu>\kappa} \right)    \right], \label{eq:LHS truncate}
\end{align}
where the last inequality follows from Lemma \ref{lmm: h simplify} and $$1- \I{Z_k^\uu>0, R_{s,k}^\uu \leq \kappa} \leq  \I{Z_k^\uu=0}+     \I{R_{s,k}^\uu>\kappa}.$$ 
Compared with \eqref{eq:LHS} in Section \ref{subsec: proof of S1}, the first two terms of \eqref{eq:LHS truncate} are similar, but the last term of \eqref{eq:LHS truncate} is extra, resulting from the hard truncation of the residual times. Since $g_n^\fk(z)\leq z^n$ for any $z\geq0$, the second term of \eqref{eq:LHS truncate} can be bounded by the constant $\Theta_1$, as detailed in \eqref{eq:indicator}. The last term of \eqref{eq:LHS truncate} can be bounded as follows:
\begin{align*}
	&\E_{\pi}\left[ \cc^{k(n-1)}  {\mu}^\uu_k g_{n}^\fk\left(\b{k}^\T {Z}^{\uu} \right)   \I{R_{s,k}^\uu\ge \kappa}     \right]\leq   \cc^{k(n-1)}   {\mu}^\uu_k \E_{\pi}\left[  (\kappa n/e)^n \I{R_{s,k}^\uu\geq \kappa}      \right] \\
    &\quad \le \left( \lambda_k + 1 \right)  (n/e)^n\E_{\pi}\left[ \left( R_{s,k}^\uu \right)^n \right]\leq  \left( \lambda_k + 1 \right)  (n/e)^n E_{1,0},
\end{align*}
where the first inequality follows from property (i) of Lemma \ref{lmm: g properties} for the truncated polynomials, the second inequality holds due to Markov's inequality for the tail of residual times, and the last inequality follows from Statement~(\hyperlink{S3}{S3}) of the induction hypotheses. As a result, we can conclude that the left-hand side of the BAR becomes
\begin{equation}\label{eq:LHS result truncate}
    -\E_{\pi}\left[\A f_{k,n}^\fk\left(X^\uu\right)  \right] \ge \frac{1}{J} \left( 1-\ww{k}{k} \right)    \E_{\pi}\left[ \cc^{kn} g_n^\fk\left(\b{k}^\T Z^{\uu} \right)   \right]  - \Theta_1 -  \left( \lambda_k + 1 \right)  (n/e)^n E_{1,0}.
\end{equation}

For each $\ell \leq M\wedge J$, the term that corresponds to jumps of external arrivals to station $\ell$ is
\begin{align}
    &\E_{e,\ell}\left[f^\fk_{k,n}\left(X^\uu+ \Delta_{e,\ell}\right) - f^\fk_{k,n}\left(X^\uu \right)  \right]\notag\\
    &\quad=\cc^{k(n-1)} \E_{e,\ell}\left[ \left(  G_n^\fk\left(\b{k}^\T  {Z}^{\uu} + \bb{k}{\ell} \right) -  G_n^\fk\left(\b{k}^\T   {Z}^{\uu}  \right)  \right)- g_n^\fk\left(\b{k}^\T  {Z}^{\uu} + \bb{k}{\ell} \right)   \bb{k}{\ell} \left(  {T}_{e,\ell}  \wedge \alpha_\ell\kappa \right)   \right]\notag\\
    &\qquad +\cc^{k(n-1)} \E_{e,\ell}\left[\left(  g_n^\fk\left(\b{k}^\T  {Z}^{\uu} + \bb{k}{\ell} \right)  -  g_n^\fk\left(\b{k}^\T   {Z}^{\uu}   \right)   \right)  h_{k} \left( R_{e}^\uu, R_{s}^\uu \right)  \right] \notag\\
    &\quad=\cc^{k(n-1)} \E_{e,\ell}\left[ \left(  G_n^\fk\left(\b{k}^\T  {Z}^{\uu} + \bb{k}{\ell} \right) -  G_n^\fk\left(\b{k}^\T   {Z}^{\uu}  \right)  \right)- g_n^\fk\left(\b{k}^\T  {Z}^{\uu} + \bb{k}{\ell} \right)   \bb{k}{\ell}   {T}_{e,\ell}    \right]\notag\\
    &\qquad +\cc^{k(n-1)} \E_{e,\ell}\left[\left(  g_n^\fk\left(\b{k}^\T  {Z}^{\uu} + \bb{k}{\ell} \right)  -  g_n^\fk\left(\b{k}^\T   {Z}^{\uu}   \right)   \right)  h_{k} \left( R_{e}^\uu, R_{s}^\uu \right)  \right]\notag\\
    &\qquad + \cc^{k(n-1)} \E_{e,\ell}\left[ g_n^\fk\left(\b{k}^\T  {Z}^{\uu} + \bb{k}{\ell} \right)   \bb{k}{\ell} \left( {T}_{e,\ell} -{T}_{e,\ell}  \wedge \alpha_\ell\kappa \right)   \right]  \label{eq:RHS1-1 truncate}
\end{align}
Compared with \eqref{eq:RHS1} in Section \ref{subsec: proof of S1}, the first two terms can similarly be bounded by the mean value theorem, as detailed in Lemma \ref{lmm: g properties}. The last term, which is an extra term, can also be bounded later. The first term of \eqref{eq:RHS1-1 truncate} becomes
\begin{align}
    &\cc^{k(n-1)} \E_{e,\ell}\left[     G_n^\fk\left(\b{k}^\T  {Z}^{\uu} + \bb{k}{\ell} \right) -  G_n^\fk\left(\b{k}^\T   {Z}^{\uu}  \right)  -g_n^\fk\left(\b{k}^\T  {Z}^{\uu} + \bb{k}{\ell} \right)   \bb{k}{\ell}    {T}_{e,\ell}   \right] \notag \\
    &\quad \leq  \cc^{k(n-1)}  \bb{k}{\ell}  \E_{e,\ell}\left[ \left(\b{k}^\T  {Z}^{\uu} +   \bb{k}{\ell} \right)^n\exp\left(-  \b{k}^\T  {Z}^{\uu}  /\kappa \right) - g_n^\fk\left(\b{k}^\T  {Z}^{\uu} + \bb{k}{\ell} \right)     \right]\notag\\
    & \quad = \cc^{k(n-1)} \bb{k}{\ell}  \E_{e,\ell}\left[ \left( \left(\b{k}^\T  {Z}^{\uu} +   \bb{k}{\ell} \right)^n - \left(\b{k}^\T  {Z}^{\uu}  \right)^n \right)\exp\left(-  \b{k}^\T  {Z}^{\uu}  /\kappa \right)   \right]  \notag \\
    & \qquad + \cc^{k(n-1)} \bb{k}{\ell}  \E_{e,\ell}\left[ g_{n}^\fk\left(  \b{k}^\T  {Z}^{\uu} \right)- g_n^\fk\left(\b{k}^\T  {Z}^{\uu} + \bb{k}{\ell} \right)   \right] \notag \\
    & \quad \leq \cc^{k(n-1)} \bb{k}{\ell}  \E_{e,\ell}\left[ \left( \left(\b{k}^\T  {Z}^{\uu} +   \bb{k}{\ell} \right)^n - \left(\b{k}^\T  {Z}^{\uu}  \right)^n \right)  + (n+1)  \bb{k}{\ell} \left(\b{k}^\T  {Z}^{\uu} + \bb{k}{\ell} \right)^{n-1}\right], \label{eq:truncate-RHS}
\end{align}
where the first inequality follows from Lemma \ref{lmm: independent}, $\E[T_{e,\ell}]=1$ and property (iii) in Lemma~\ref{lmm: g properties}, the last inequality holds due to $x\leq |x|$ for any $x$ and property (ii) in Lemma \ref{lmm: g properties}. According to Statement (\hyperlink{S2}{S2}) of the induction hypotheses, the final term in \eqref{eq:truncate-RHS} can also be bounded by a constant $\Theta_9$, which is independent of $r \in (0, r_0)$, $\kappa>0$ and $\ell \leq M \wedge J$. Similarly, the second term in \eqref{eq:RHS1-1 truncate} can be written as
\begin{align}
    &\cc^{k(n-1)} \E_{e,\ell}\left[\left( g_n^\fk \left(\b{k}^\T  {Z}^{\uu} + \bb{k}{\ell} \right)  - g_n^\fk  \left(\b{k}^\T   {Z}^{\uu}   \right)   \right)  h_{k} \left( R_{e}^\uu, R_{s}^\uu \right)  \right] \notag \\
    &\quad \leq (n+1)\bb{k}{\ell}  \E_{e,\ell}\left[ \cc^{k(n-1)} \left(\b{k}^\T  {Z}^{\uu} +  \bb{k}{\ell} \right)^{n-1}  \abs{h_{k}\left({R}_e^\uu, {R}_s^\uu\right)}  \right]. \label{eq:RHS1-2 truncate}
\end{align}
According to Statement~(\hyperlink{S4}{S4}) of the induction hypotheses, the final term in \eqref{eq:RHS1-2 truncate} can be bounded by a constant  $\Theta_{10}$, which is independent of $r\in (0,r_0)$, $\kappa>0$ and $\ell \leq M\wedge J$. The last term of \eqref{eq:RHS1-1 truncate} can be bounded as follows:
\begin{align}
    &\cc^{k(n-1)} \E_{e,\ell}\left[  g_{n}^\fk\left(\b{k}^\T  {Z}^{\uu} + \bb{k}{\ell} \right) \left(  {T}_{e,\ell}-  {T}_{e,\ell}\wedge \kappa \alpha_\ell \right)\right] \notag\\
    &\quad\leq \E_{e,\ell}\left[   g_{n}^\fk\left(\b{k}^\T  {Z}^{\uu} + \bb{k}{\ell} \right)    {T}_{e,\ell} \I{  {T}_{e,\ell}\geq  \kappa \alpha_\ell  }  \right]\leq \E \left[  (\kappa n/e)^n  {T}_{e,\ell} \I{  {T}_{e,\ell}\geq  \kappa \alpha_\ell  }  \right] \notag \\
    &\quad \leq  (n/e)^n\E \left[   {T}^{n+1}_{e,\ell} /\alpha_\ell^n \I{  {T}_{e,\ell}\geq  \kappa \alpha_\ell  }  \right]\leq  (n/e)^n\E \left[   {T}^{n+1}_{e,\ell}   \right]/\alpha_\ell^n, \label{eq:RHS1 time}
\end{align}
where the second inequality follows from property (i) of Lemma \ref{lmm: g properties}.
Therefore, we have
\begin{equation} \label{eq:RHS1 result truncate}
    \E_{e,\ell}\left[ f_{k,n} \left(X^\uu+ \Delta_{e,\ell}\right) - f_{k,n} \left(X^\uu\right) \right]\leq \Theta_9 + \Theta_{10}+ (n/e)^n\E \left[   {T}^{n+1}_{e,\ell}   \right]/\alpha_\ell^n.
\end{equation} 
As $f_{k,n}^\fk$ is independent of states affected by external arrivals to station $\ell>M\wedge J$, the terms corresponding to jumps of external arrivals to station $\ell>M\wedge J$ are all zero.

For each $\ell\in \J$, the term corresponding to jumps of service completions at station $\ell$ is
\begin{align}
    &\E_{s,\ell}\left[ f_{k,n} \left(X^\uu+ \Delta_{s,\ell}\right) - f_{k,n} \left(X^\uu\right) \right]\notag\\
    &=  {\cc^{k(n-1)}} \E_{s,\ell}\left[\left(G_n^\fk\left(\b{k}^\T {Z}^{\uu} + \Delta_Z \right)  -   G_n^\fk\left(\b{k}^\T   {Z}^{\uu}   \right) \right)   +     g_n^\fk\left(\b{k}^\T   {Z}^{\uu} + \Delta_Z \right)    \left( \bb{k}{\ell} - \ww{\ell}{k} \right)   T_{s,\ell}     \right]\notag\\
    &\quad +\cc^{k(n-1)}  \E_{s,\ell}\left[ \left( g_n^\fk \left(\b{k}^\T   {Z}^{\uu} + \Delta_Z \right)  -  g_n^\fk\left(\b{k}^\T   {Z}^{\uu}   \right)   \right) h_{k}\left({R}_e^\uu, {R}_s^\uu\right)  \right]\notag\\
    &\quad +  {\cc^{k(n-1)}} \E_{s,\ell}\left[ g_n^\fk\left(\b{k}^\T   {Z}^{\uu} + \Delta_Z \right)^n    \left( -\bb{k}{\ell} + \ww{\ell}{k} \right)  \left(  T_{s,\ell} - T_{s,\ell} \wedge \kappa \mu_\ell^\uu \right)     \right]  \label{eq: RHS2 truncate}
\end{align} 
Compared with \eqref{eq:RHS2} in Section \ref{subsec: proof of S1}, the first two terms can be bounded using the mean value theorem, as similarly demonstrated in \eqref{eq:truncate-RHS} and \eqref{eq:RHS1-2 truncate}. The last term can be similarly bounded following \eqref{eq:RHS1 time}. Consequently, we can conclude that \eqref{eq: RHS2 truncate} can be bounded by a constant $\Theta_{11}$, which is independent of $r \in (0, r_0)$, $\kappa>0$ and $\ell \leq \J$:
\begin{equation} \label{eq:RHS2 result truncate}
    \E_{s,\ell}\left[ f_{k,n} \left(X^\uu+ \Delta_{s,\ell}\right) - f_{k,n} \left(X^\uu\right) \right]\leq \Theta_{11}.
\end{equation} 

In summary, it follows from \eqref{eq:LHS result truncate}, \eqref{eq:RHS1 result truncate}, \eqref{eq:RHS2 result truncate}, and BAR \eqref{eq:BAR} that
\begin{align*} 
    \frac{1}{J} \left( 1-\ww{k}{k} \right)   & \E_{\pi}\left[ \cc^{kn} g_n^\fk\left(\b{k}^\T Z^{\uu} \right)   \right]  - \Theta_1 -  \left( \lambda_k + 1 \right)  (n/e)^n E_{1,0} \\
    & \leq \sum_{\ell=1}^{M\wedge J}\alpha_\ell\left( \left( \Theta_9 + \Theta_{10} \right) +  (n/e)^n\E \left[   {T}^{n+1}_{e,\ell}   \right]/\alpha_\ell^n \right) + \sum_{\ell=1}^{J}\Theta_{11} \lambda_\ell.
\end{align*}
By letting $\kappa$ go to infinity and utilizing the monotone convergence theorem, we achieve a result analogous to \eqref{eq: result S1}, confirming that $\E_{\pi}[ (\cc^{k} u' {Z}^{\uu} )^n ]$ is uniformly bounded. Consequently,  $\E_{\pi}[ (\cc^{k}  {Z}^{\uu}_k )^n   ] $ is uniformly bounded.

\subsection{Extension to Non-integer Cases}\label{sec: beta general}
We extend Theorem \ref{thm:Ckn} to the case where $M$ is a real number in $[1,\infty)$.

\begin{theorem}\label{thm:general}
    Given an integer $M\geq 1$ and a constant $\varepsilon\in (0,1)$, suppose the following moments exist for the unitized times:  
    \begin{equation}\label{eq: moment condition general}
        \E\left[T_{e,j}^{ M+1+\varepsilon  }\right]<\infty  \text{ for } 1\leq j \leq M\wedge J, \text{ and }~\E\left[T_{s,j}^{ M+1+\varepsilon}\right]<\infty\text{ for  }1\leq j \leq J.
    \end{equation}
    Then, for each $1\leq k\leq M\wedge J$, there exists a positive constant $C_{k}<\infty$ such that for all $\cc\in(0,r_0)$,
    \begin{equation*}
        \E_\pi\left[ \left( \cc^k Z_{k}^\uu  \right)^{\beta}\right] \leq C_{k}, 
    \end{equation*}
    where $\beta\equiv M+\varepsilon/(M+\varepsilon)$ and $r_0$ is a constant $\in (0,1)$, specified in Lemma \ref{lmm: eta}.
\end{theorem}

To prove Theorem \ref{prop:general},  it is essential to introduce the following proposition, which aligns directly with Statements (\hyperlink{S1}{S1})-(\hyperlink{S4}{S4}) in Section \ref{subsec: test functions } by replacing $M$ with $\beta$ and substituting $n$ with $\beta-1$.
\begin{proposition}\label{prop:general}
    Under the moment condition \eqref{eq: moment condition general}, there exist positive and finite constants $C_{k,\beta-1}$, $D_{k,\beta-1}$, $E_{k,\beta-1}$ and $F_{k,\beta-1}$ independent of $\cc$ such that the following statements hold for all $1\leq k\leq M\wedge J$ and $\cc\in (0,r_0)$: 
    \begin{enumerate}
        \item[\hypertarget{1}{(1)}] ~$
        \E_{\pi}\left[\left(\cc^kZ_k^\uu\right)^{\beta-1}\right]\leq C_{k,\beta-1}$;
    \item[\hypertarget{2}{(2)}] 
    $\sum\limits_{\ell=1}^{M\wedge J}\E_{e,\ell}\left[\left(\cc^kZ_k^\uu\right)^{\beta-1}\right]+ \sum\limits_{\ell=1}^J\E_{s,\ell}\left[\left(\cc^kZ_k^\uu\right)^{\beta-1}\right] \leq D_{k,\beta-1}$;
    \item[\hypertarget{3}{(3)}] 
        ~$\E_{\pi} \left[ \left( \cc^k Z_{k}^{\uu}\right)^{\beta-1}\psi_{1}\left(R_e^{(r)},R_s^{(r)}\right) \right] \leq E_{k,\beta-1}$;
    \item[\hypertarget{4}{(4)}]  
    $\sum\limits_{\ell=1}^{M\wedge J} \E_{e,\ell}\left[ \left( \cc^k Z_{k}^{\uu}\right)^{\beta-1} \psi_{1}\left(R_e^{(r)},R_s^{(r)}\right)\right]+\sum\limits_{\ell=1}^J \E_{s,\ell}\left[ \left( \cc^k Z_{k}^{\uu}\right)^{\beta-1} \psi_{1}\left(R_e^{(r)},R_s^{(r)}\right)\right] \leq F_{k,\beta-1}$;
    \end{enumerate}
    where $\psi_{1}$ is given in \eqref{eq: function psi}.
\end{proposition}
\begin{proof}[Proof of Theorem \ref{thm:general}]
    The proof utilizes the mathematical induction on station $k$ ranging from $1$ to $M\wedge J$. Specifically, for station $k$, we assume that $\E_\pi[ ( \cc^j Z_{j}^\uu  )^{\beta}]$ are uniformly bounded for all $r\in(0,r_0)$ and stations $1\leq j\leq k-1$. Under this assumption, the proof of Theorem \ref{thm:general} for station $k$ parallels the approach in Section \ref{subsec: proof of S1}, with only two changes: substituting $n$ with $\beta$ and replacing Statements (\hyperlink{S1}{S1})-(\hyperlink{S4}{S4}) with Statements (\hyperlink{1}{1})-(\hyperlink{4}{4}).
\end{proof}
To prove Proposition \ref{prop:general}, we need to introduce the following lemma, which directly corresponds to the base step of Statements (\hyperlink{S3}{S3})-(\hyperlink{S4}{S4}) when $n=0$.
\begin{lemma}\label{lmm: general}
    Under the moment condition \eqref{eq: moment condition general}, there exist positive and finite constants $A_1$ and $A_2$ independent of $\cc$ such that the following statements hold for all $1\leq k\leq M\wedge J$ and $\cc\in (0,r_0)$: 
    \begin{align}
        &\E_{\pi} \left[\psi_{M+\varepsilon}\left(R_e^{(r)},R_s^{(r)}\right) \right] \leq A_1 \label{eq: S3 general}\\
        &\sum\limits_{\ell=1}^{M\wedge J} \E_{e,\ell}\left[  \psi_{M+\varepsilon}\left(R_e^{(r)},R_s^{(r)}\right)\right]+\sum\limits_{\ell=1}^J \E_{s,\ell}\left[  \psi_{M+\varepsilon}\left(R_e^{(r)},R_s^{(r)}\right)\right] \leq A_2\label{eq: S4 general}
    \end{align}
\end{lemma}

\begin{proof}[Proof of Lemma \ref{lmm: general}]
    The proof of Lemma \ref{lmm: general} follows the same approach as presented in Section \ref{subsec: bounds for residual times}. Specifically, \eqref{eq: S3 general} follows from Lemma~6.4 of \citet{BravDaiMiya2023} and the derivation of \eqref{eq: E10}. Furthermore, the proof of \eqref{eq: S4 general} can directly adopt the proof approach of Statement (\hyperlink{S4}{S4}) in Section~\ref{subsec: bounds for residual times} by replacing $M$ with $M+\varepsilon$ in the test function \eqref{eq:func base} and its subsequent derivation.
\end{proof}

\begin{proof}[Proof of Proposition \ref{prop:general}]
    Since the moment condition \eqref{eq: moment condition general} implies \eqref{eq: moment condtion}, Statements (\hyperlink{S1}{S1}) and (\hyperlink{S2}{S2}) hold for $n=M$, and hence, Statements (\hyperlink{1}{1}) and (\hyperlink{2}{2}) also hold due to $\beta-1<M$. 
    Statement (\hyperlink{3}{3}) can be established as follows:
    \begin{equation} \label{eq: proof S3 general}
        \E_{\pi}\left[ \left( \cc^k Z_{k}^{\uu}\right)^{\beta-1} \psi_{1}\right] \leq \E_{\pi}\left[ \left( \cc^k Z_{k}^{\uu}\right)^{M} \right]^{1-\frac{1}{M+\varepsilon}} \E_{\pi}\left[ \psi_{1}^{M+\varepsilon}\right]^{\frac{1}{M+\varepsilon}} \leq C_{k,M}^{1-\frac{1}{M+\varepsilon}}\cdot 2J A_1^{\frac{1}{M+\varepsilon}},
    \end{equation}
    where the first inequality holds due to Hölder's inequality, and the second inequality follows from Lemma \ref{lmm: general} and the fact that $\psi_{1}^{M+\varepsilon}  \leq 2J ( R_{\max}^\uu )^{M+\varepsilon}  \leq  2J \psi_{M+\varepsilon}$ with $R_{\max}^\uu\equiv \max(\{R_{e,j}^{(r)}: 1\leq j\leq M\wedge J\}\cup\{R_{s,j}^{(r)}: 1\leq j\leq J\})$. Consequently, $E_{k,\beta-1}$ can be set as $2JC_{k,M}^{1-\frac{1}{M+\varepsilon}} A_1^{\frac{1}{M+\varepsilon}}$. Statement (\hyperlink{4}{4}) follows the same argument as \eqref{eq: proof S3 general} by replacing $\pi$ to any Palm measure $\nu$ in $\{\nu_{e,\ell};1\leq \ell \leq M\wedge J\}\cup \{\nu_{s,\ell};\ell\in \J\}$:
    \begin{equation*}
        \E_{\nu}\left[ \left( \cc^k Z_{k}^{\uu}\right)^{\beta-1} \psi_{1}\right] \leq \E_{\nu}\left[ \left( \cc^k Z_{k}^{\uu}\right)^{M} \right]^{1-\frac{1}{M+\varepsilon}} \E_{\nu}\left[ \psi_{1}^{M+\varepsilon}\right]^{\frac{1}{M+\varepsilon}} \leq D_{k,M}^{1-\frac{1}{M+\varepsilon}}\cdot 2J A_2^{\frac{1}{M+\varepsilon}}.
    \end{equation*}
    Consequently, $F_{k,\beta-1}$ can be set as $4J^2D_{k,M}^{1-\frac{1}{M+\varepsilon}} A_2^{\frac{1}{M+\varepsilon}}$.
\end{proof}

\bibliographystyle{plainnat} 
\bibliography{reference}

\end{document}